\documentclass[11pt,reqno]{amsart}

\usepackage[utf8]{inputenc}
\usepackage[T1]{fontenc}
\usepackage{amsmath,amssymb,amsthm,mathtools}
\usepackage{graphicx}
\usepackage{enumitem}
\usepackage[margin=3cm]{geometry}
\usepackage{hyperref}

\def\h{\mathbb H}
\newcommand{\R}{\mathbb R}

\DeclareMathOperator{\artanh}{artanh}

\theoremstyle{plain}
\newtheorem{theorem}{Theorem}[section]
\newtheorem{proposition}[theorem]{Proposition}
\newtheorem{lemma}[theorem]{Lemma}
\newtheorem{corollary}[theorem]{Corollary}
\theoremstyle{definition}
\newtheorem{definition}[theorem]{Definition}

\newtheorem{remark}[theorem]{Remark}

\begin{document}

\title[Solitons of the ICSF in the hyperbolic plane]
{Solitons for the inverse curve shortening flow\\
on the hyperbolic plane}

\author{Ivan Krznari\'c}
\address{Faculty of Economics and Business, University of Zagreb}
\email{ikrznaric@efzg.hr}

\author{Rafael L\'opez}
\address{Departamento de Geometr\'ia y Topolog\'ia, Universidad de Granada,
18071 Granada, Spain}
\email{rcamino@ugr.es}


\begin{abstract}
We study the solitons of the inverse curve shortening flow in the hyperbolic
plane with respect to Killing fields: all parabolic fields, and the hyperbolic
and elliptic fields with unit translation length and unit angular speed. The
three cases reduce to one planar system. Except for a family of horocycles,
every soliton has an endpoint at finite distance where the curvature blows up
and the curve is a branch of a cusp. In the parabolic and elliptic cases there
is a critical value $H$ of the parameter of the family: below $H$ the other end
meets the ideal boundary orthogonally, above $H$ it is a second singular
endpoint, and at $H$ the curvature tends to $1$. In the hyperbolic case
$H=\pi/2$ and the critical solitons are horocycles. These horocycles are the
only complete solitons and the only solitons of constant curvature.
\end{abstract}

\maketitle

%%%%%%%%%%%%%%%%%%%%%%%%%%%%%%%%%%%%%%%%%%%%%%%%%%%%%%%%%%%%%%%%%%%%%%%%%%%%
\section{Introduction}\label{sec:intro}

Let $\gamma\colon(a,b)\to\h^2$ be a curve in the hyperbolic plane. If $J\subset\R$
is an interval with $0\in J$, a one-parameter family
$\{\gamma_t\colon(a,b)\to\h^2,\ t\in J\}$ evolves by the \emph{inverse curve
shortening flow} (ICSF) with initial condition $\gamma$ if
\begin{equation}\label{flow}
\left\{
\begin{split}
\Big(\frac{\partial\gamma_t}{\partial t}\Big)^{\!\perp}&=-\frac{1}{\kappa(t,\cdot)}\,N(t,\cdot),\\
\gamma_0&=\gamma,
\end{split}\right.
\end{equation}
where $\kappa(t,\cdot)$ and $N(t,\cdot)$ are the curvature and unit normal of
$\gamma_t$. The ICSF is the one-dimensional case of the inverse mean curvature
flow (IMCF), and the counterpart of the curve shortening flow (CSF) \cite{ga},
whose speed is $\kappa$. The solitons of the CSF in $\h^2$ are studied in
\cite{nt} and its self-similar solutions in \cite{wx}. The CSF and its inverse
in the light cone of Minkowski space are studied in \cite{dst}, and the inverse
curvature flow in the de Sitter plane in \cite{ti}. Kwong, Wei and
the Wheelers \cite{kw} study, for closed convex curves in two-dimensional space
forms, the flow with speed $\phi'(r)/\kappa-\langle V,N\rangle$, where $V$ is a
radial conformal Killing field, and prove convergence to geodesic circles. In
Euclidean space, Gerhardt \cite{ge} and Urbas \cite{ur} proved independently
that compact star-shaped hypersurfaces with positive mean curvature converge,
after rescaling, to round spheres under the IMCF. See also \cite{cd,ch}, and
\cite{hi} for the weak formulation. Non-compact solitons of the IMCF are
constructed in \cite{dhw,hi}, and the self-similar solutions of the ICSF in
$\R^2$ are studied in \cite{chg}. The IMCF has been studied in hyperbolic space
\cite{bda,ge2,ge3,hw} and in other ambient spaces
\cite{cl,kp,lw,lu,pi1,pi2,sc,ns,zh}. Solitons of the mean curvature flow in
hyperbolic space are studied in \cite{bl1}, and in $\h^2\times\R$ in \cite{bl2}.

A curve $\gamma$ is a soliton with respect to a vector field $X$ if
$\Phi_t\circ\gamma$ solves \eqref{flow}, where $\Phi_t$ is the flow of $X$. If
$X$ is a Killing field, this is equivalent to $1/\kappa=-\langle N,X\rangle$
(Proposition~\ref{prop:soliton}). Up to conjugation there are three
one-parameter groups of isometries of $\h^2$, and each one is the group of
translations in $q$ of a metric
\[
\langle,\rangle=\frac{dq^2+dp^2}{f(p)^2},
\]
with $f(p)=p$ (parabolic), $f(p)=\sin p$ (hyperbolic) and $f(p)=\sinh p$
(elliptic). In the elliptic case this metric lives on the universal cover of
$\h^2\setminus\{i\}$ (Proposition~\ref{prop:models}). In the parabolic class
every field $c\,\partial_x$, $c>0$, is conjugate to $\partial_x$. In the other
two classes $c$ cannot be removed, and we take $c=1$, so that the horocycles
that expand under the ICSF are solitons (Proposition~\ref{prop:evol} and
Remark~\ref{rem:c}). Let $h=f'/f$. In terms of the Euclidean arc length of the
$(q,p)$-plane and the angle $\theta$ of the tangent with $\partial_q$, the
soliton equation for $X=\partial_q$ becomes
\begin{equation}\label{eq:main}
p'=\sin\theta,\qquad
\theta'=\frac{1}{\sin\theta}-h(p)\cos\theta ,
\end{equation}
and the curvature is
\begin{equation}\label{eq:kappa}
\kappa=\frac{f(p)}{\sin\theta}.
\end{equation}

Every maximal soliton (Definition~\ref{def:maximal}) has positive curvature and
meets each level set of $p$ at most once: each horizontal line in the parabolic
case, each ray from $0$ in the hyperbolic case and each hyperbolic circle centred
at $i$ in the elliptic case. Its tangent is orthogonal to $X$ at exactly one
point (Theorem~\ref{th:normal} and Lemma~\ref{lem:hypnormal}), and the value
$p_0$ of $p$ at that point determines the soliton up to isometries preserving
$X$.

A \emph{singular endpoint} is an end that converges to a point of $\h^2$, where
the limiting tangent is parallel to $X$ and $\kappa\to+\infty$. Every maximal soliton has a
singular endpoint, except the horocycles $y=b(x^2+y^2)$, $b>0$, of the
hyperbolic class (Theorems~\ref{th:upper}, \ref{th:horo} and~\ref{th:hypends}).
Hence these horocycles are the only complete solitons, and they are also the
only solitons of constant curvature (Corollary~\ref{cor:complete}). Similarly,
translating solitons of the IMCF in Euclidean space are incomplete
\cite[Theorem~3.3]{kp}. Near a singular endpoint the soliton is a branch of the
cusp of the cycloid of radius $1/4$, the translating soliton of the ICSF in
$\R^2$ \cite{dhw,kp}, and, as for the cycloid, the curve continues across the
cusp as another soliton (Remark~\ref{rem:cusps}).

In the parabolic and elliptic classes there is $H>0$ such that the other end
behaves as follows (Theorem~\ref{th:class}). If $p_0<H$, it meets the ideal
boundary orthogonally at a point. If $p_0>H$, it is a second singular endpoint,
and the soliton is an arc of finite length with compact closure. If $p_0=H$, then
$\kappa\to1$ and $q\to-\infty$ along it. In the parabolic case this end converges
to the ideal point $\infty$, with $y\sim Ke^{x}$ as $x\to-\infty$. In the
elliptic case it winds infinitely many times around $i$ and accumulates on the
whole ideal boundary. We prove $\tfrac12<H\le\tfrac12\big(1+\sqrt{\pi/2}\big)^2$
in the parabolic case and $\artanh\tfrac12<H<3.162$ in the elliptic case.
Numerically, $H\approx1.75694$ and $H\approx2.02556$.

In the hyperbolic class $p$ is the polar angle, and the value $\pi/2$ plays the
role of $H$. The solitons with $p_0<\pi/2$ meet the ideal boundary orthogonally,
the one with $p_0=\pi/2$ is a horocycle, whose two ends converge to the ideal
point $0$, and those with $p_0>\pi/2$ are the reflections of the first ones in
the $y$-axis (Theorem~\ref{th:hyp}). Thus every hyperbolic soliton other than a
horocycle has one singular endpoint and one end on the ideal boundary. In
particular, there are no solitons of finite length in this class.

The proof of Proposition~\ref{prop:soliton} uses that the flow of $X$ consists
of isometries. In contrast, for the conformal field $\partial_p$ no curve moves
under the ICSF by the flow of the field (Proposition~\ref{prop:conf}).

Solitons of the ICSF in $\h^2$ also appear in recent work of Nunes da Silva
and Salinas Reyes \cite{nsr}, who obtain them from solitons of the curvature
flow in the de Sitter plane through the normal map, and prove that the
curvature at each end is unbounded or tends to $-1$, $0$ or $1$. Our approach
works directly in $\h^2$ and gives, for each soliton, the geometry of its two
ends, the critical value $H$ that separates the three behaviours, and the
characterisation of the complete solitons.

%%%%%%%%%%%%%%%%%%%%%%%%%%%%%%%%%%%%%%%%%%%%%%%%%%%%%%%%%%%%%%%%%%%%%%%%%%%%
\section{Preliminaries}\label{sec:prelim}

We work in the upper half-plane model
$\R^2_+=\{(x,y)\in\R^2\colon y>0\}$ with the metric
$\langle,\rangle=(dx^2+dy^2)/y^2$. The ideal boundary is
$\h^2_\infty=\{y=0\}\cup\{\infty\}$.

For a regular curve $\gamma(s)=(x(s),y(s))$ in $\R^2_+$, the unit normal and the
hyperbolic curvature are
\begin{equation}\label{ke}
N=y\,\frac{(-y',x')}{\sqrt{x'^2+y'^2}},\qquad
\kappa=y\,\kappa_e+\frac{x'}{\sqrt{x'^2+y'^2}},\qquad
\kappa_e=\frac{x'y''-x''y'}{(x'^2+y'^2)^{3/2}} .
\end{equation}
Curves of constant hyperbolic curvature are the intersections with $\R^2_+$ of
Euclidean lines and circles. They are geodesics ($\kappa=0$), equidistant lines
($0<\kappa<1$), horocycles ($\kappa=1$) or hyperbolic circles ($\kappa>1$).

Let $X\in\mathfrak X(\h^2)$ be a Killing field and $\Phi\colon\h^2\times\R\to\h^2$
its flow, so that $\frac{d}{dt}\Phi(w,t)=X(\Phi(w,t))$ and $\Phi(w,0)=w$.

\begin{definition}
A curve $\gamma$ is a \emph{soliton} of the ICSF with respect to $X$ if the
family $\gamma_t:=\Phi_t\circ\gamma$ is a solution of \eqref{flow}.
\end{definition}

\begin{proposition}\label{prop:soliton}
Let $X$ be a Killing field. A curve $\gamma$ is a soliton of the ICSF with
respect to $X$ if and only if
\begin{equation}\label{soliton}
\frac{1}{\kappa}=-\langle N,X\circ\gamma\rangle .
\end{equation}
\end{proposition}

\begin{proof}
Each $\Phi_t$ is an isometry, so the unit normal of $\gamma_t$ is
$N(t,s)=d\Phi_t(N(s))$ and $\kappa(t,s)=\kappa(s)$, where $N$ and $\kappa$ refer
to $\gamma=\gamma_0$. Since
$\frac{\partial\gamma_t}{\partial t}=X(\gamma_t)$,
\[
-\Big\langle\frac{\partial\gamma_t}{\partial t},N(t,s)\Big\rangle
=-\big\langle X(\Phi_t(\gamma(s))),d\Phi_t(N(s))\big\rangle
=-\big\langle X(\gamma(s)),N(s)\big\rangle ,
\]
where we used that $\Phi_t$ is an isometry and $d\Phi_t(X)=X\circ\Phi_t$. Hence
the normal component of \eqref{flow} holds for every $t$ if and only if it holds
for $t=0$, which is \eqref{soliton}.
\end{proof}

The space of Killing fields of $\h^2$ is three-dimensional. In the complex
notation $z=x+iy$, a basis is
\begin{equation}\label{kv}
X_1=\partial_x,\qquad
X_2=x\partial_x+y\partial_y,\qquad
X_3=(x^2-y^2)\partial_x+2xy\partial_y .
\end{equation}
They are the infinitesimal generators of the groups
$z\mapsto z+t$, $z\mapsto e^tz$ and $z\mapsto z/(1-tz)$. The isometry $w=-1/z$
conjugates the third group to $w\mapsto w+t$, so $X_3$ is conjugate to $X_1$. The
rotations about $i$ are $z\mapsto(\cos a\,z+\sin a)/(-\sin a\,z+\cos a)$, and the
generator of this group, for the parameter $\alpha=2a$, is $\frac12(X_1+X_3)$.

Up to conjugation there are three types of one-parameter groups of isometries
of $\h^2$: parabolic, hyperbolic and elliptic. They are represented by $X_1$,
$X_2$ and $X_1+X_3$. Accordingly we call a curve a \emph{parabolic},
\emph{hyperbolic} or \emph{elliptic} soliton of the ICSF. Equation
\eqref{soliton} is invariant under isometries, but it is not homogeneous in $X$.
For the parabolic class, $c\,\partial_x$ is conjugate to $\partial_x$ by the
homothety $z\mapsto cz$. For the hyperbolic and elliptic classes, the
translation length and the rotation speed are conjugacy invariants, so we fix
the normalisations
\begin{equation}\label{eq:fields}
X_{\mathrm{par}}=\partial_x,\qquad
X_{\mathrm{hyp}}=x\partial_x+y\partial_y,\qquad
X_{\mathrm{ell}}=\tfrac12\big[(1+x^2-y^2)\partial_x+2xy\partial_y\big].
\end{equation}
With these normalisations the parameter of the flow is the translation length
along the axis $\{x=0\}$ for $X_{\mathrm{hyp}}$, and the angle of rotation about
$i$ for $X_{\mathrm{ell}}$; this is the reason for the factor $\tfrac12$. The
solitons of $-X$ are the images of those of $X$ under an isometry conjugating $X$
to $-X$: $x\mapsto-x$ for $X_{\mathrm{par}}$, $z\mapsto-1/z$ for
$X_{\mathrm{hyp}}$ and a reflection for $X_{\mathrm{ell}}$. For
$cX_{\mathrm{hyp}}$ and $cX_{\mathrm{ell}}$ with $c\neq1$ the equation is
different (Remark~\ref{rem:c}).

The ICSF is not defined for geodesics, since $\kappa=0$. For the other curves of
constant curvature we have the following solutions.

\begin{proposition}\label{prop:evol}
Let $\gamma$ be a curve of constant curvature $\kappa>0$. After an isometry, the
following families are solutions of \eqref{flow} with $\gamma_0=\gamma$.
\begin{enumerate}[label=\textup{(\arabic*)}]
\item If $\kappa>1$, let $\coth R=\kappa$ and
$\gamma_t(s)=\sinh r(t)\,(\cos s,\sin s)+(0,\cosh r(t))$, with
$r(t)=\sinh^{-1}(\sinh R\,e^{t})$ and $t\in\R$. These are the hyperbolic circles
of centre $(0,1)$. They expand as $t\to+\infty$ and shrink to the centre as
$t\to-\infty$.
\item If $\kappa=1$, let $\gamma_t(s)=r(t)\,(\cos s,\sin s)+(0,r(t))$, with
$r(t)=Re^t$ and $t\in\R$. These are horocycles with ideal point $0$, and they
expand.
\item If $0<\kappa<1$, let $a>0$, $R=a/\sqrt{1-\kappa^2}$ and
$\gamma_t(s)=r(t)\,(\cos s,\sin s)+\big(0,\sqrt{r(t)^2-a^2}\big)$, with
$r(t)=Re^t$, $t>\log(a/R)$ and $\sin s>-\sqrt{r^2-a^2}/r$. These are the
equidistant lines of the geodesic $|z|=a$. They converge to the geodesic as
$t\to\log(a/R)$ and expand as $t\to+\infty$.
\end{enumerate}
\end{proposition}

\begin{proof}
After an isometry, in the three cases $\gamma_t$ is an arc of the Euclidean circle
$\varrho(\cos s,\sin s)+(0,b)$, with $\varrho=\varrho(t)$ and $b=b(t)$, and
$y=\varrho\sin s+b$. By \eqref{ke}, for the normal $N=-y(\cos s,\sin s)$ we have
$\kappa=b/\varrho$ and
\[
\langle\partial_t\gamma_t,N\rangle=-\frac{\varrho'+b'\sin s}{y}.
\]
\begin{enumerate}
\item We have $\varrho=\sinh r$, $b=\cosh r$ and $\varrho'+b'\sin s=r'y$, so
$\kappa=\coth r$ and $\langle\partial_t\gamma_t,N\rangle=-r'$. Then \eqref{flow}
is $r'=\tanh r$, whose solution with $r(0)=R$ is the stated one.
\item We have
$\varrho=b=r$ and $\varrho'+b'\sin s=r'y/r$, so $\kappa=1$ and
$\langle\partial_t\gamma_t,N\rangle=-r'/r$. Then \eqref{flow} is $r'=r$.
\item We have $\varrho=r$, $b=\sqrt{r^2-a^2}$ and $\varrho'+b'\sin s=r'y/b$, so
$\kappa=b/r$ and $\langle\partial_t\gamma_t,N\rangle=-r'/b$. Then \eqref{flow}
is again $r'=r$, and $b$ is real for $t>\log(a/R)$. The initial curve has
curvature $b(0)/R=\kappa$.
\end{enumerate}
\end{proof}

%%%%%%%%%%%%%%%%%%%%%%%%%%%%%%%%%%%%%%%%%%%%%%%%%%%%%%%%%%%%%%%%%%%%%%%%%%%%
\section{The soliton system}\label{sec:unified}

\begin{proposition}\label{prop:models}
Let $f$ be one of $f(p)=p$ on $I=(0,\infty)$, $f(p)=\sin p$ on $I=(0,\pi)$, or
$f(p)=\sinh p$ on $I=(0,\infty)$, and let
\[
M_f=\big(\R\times I,\ \tfrac{dq^2+dp^2}{f(p)^2}\big).
\]
For $f=p$ and $f=\sin p$, $M_f$ is isometric to $\h^2$. For $f=\sinh p$, $M_f$
is isometric to the universal cover of $\h^2\setminus\{i\}$, with $q$ a lift of
the angle about $i$. In the three cases the translations $q\mapsto q+t$
correspond to the flow of $X_{\mathrm{par}}$, $X_{\mathrm{hyp}}$ and
$X_{\mathrm{ell}}$ of \eqref{eq:fields}, respectively.
\end{proposition}

\begin{proof}
For $f(p)=p$ this is the upper half-plane itself with $(q,p)=(x,y)$ and
$X_1=\partial_x$.

For $f(p)=\sin p$, write $x=r\cos\varphi$, $y=r\sin\varphi$ with $r>0$,
$\varphi\in(0,\pi)$, and set $u=\log r$. Then
\[
\frac{dx^2+dy^2}{y^2}=\frac{dr^2+r^2d\varphi^2}{r^2\sin^2\varphi}
=\frac{du^2+d\varphi^2}{\sin^2\varphi},
\]
and $X_{\mathrm{hyp}}=x\partial_x+y\partial_y=r\partial_r=\partial_u$. So
$(q,p)=(u,\varphi)$.

For $f(p)=\sinh p$, use geodesic polar coordinates $(\rho,\nu)$ about $i$, in
which the metric is $d\rho^2+\sinh^2\!\rho\,d\nu^2$ and $X_{\mathrm{ell}}
=\partial_\nu$. Put $t=-\log\tanh(\rho/2)$, i.e.\ $\rho=2\artanh(e^{-t})$. Then
$\sinh\rho=1/\sinh t$ and $d\rho=-dt/\sinh t$, so
\[
d\rho^2+\sinh^2\!\rho\,d\nu^2=\frac{dt^2+d\nu^2}{\sinh^2 t},
\]
and $(q,p)=(\nu,t)$, where $\nu\in\R$ is a lift of the angle and
$t\in(0,\infty)$ corresponds to $\rho\in(0,\infty)$, so the point $i$ is not
covered. That $\partial_\nu=X_{\mathrm{ell}}$ in half-plane
coordinates follows from $z=i\frac{1+w}{1-w}$, $w=\tanh(\rho/2)e^{i\nu}$: the
generator $\dot w=iw$ of the disc rotations transforms into
$\dot z=\tfrac12(1+z^2)$.
\end{proof}

Let $f$ be as above, $h:=f'/f$, and let $\gamma(s)=(q(s),p(s))$ be a curve in
$M_f$ parametrised by Euclidean arc length in the $(q,p)$-plane, so that
$q'=\cos\theta$, $p'=\sin\theta$ for a smooth angle function $\theta$. Since the
metric is conformal to the Euclidean one with factor $\lambda=1/f(p)$, we have
\begin{equation}\label{eq:Nk}
N=f(p)\,(-\sin\theta,\cos\theta),\qquad
\kappa=f(p)\,\theta'+f'(p)\cos\theta,\qquad
\langle N,\partial_q\rangle=-\frac{\sin\theta}{f(p)} .
\end{equation}
Thus \eqref{soliton} is equivalent to
\begin{equation}\label{eq:kap}
\kappa=\frac{f(p)}{\sin\theta}.
\end{equation}

\begin{proposition}\label{prop:system}
A curve $\gamma$ in $M_f$ is a soliton of the ICSF with respect to
$X=\partial_q$ if and only if
\begin{equation}\label{sys}
\left\{
\begin{split}
q'&=\cos\theta,\\
p'&=\sin\theta,\\
\theta'&=\frac{1}{\sin\theta}-h(p)\cos\theta .
\end{split}\right.
\end{equation}
\end{proposition}

The coordinate $q$ is obtained by a quadrature, so \eqref{sys} is equivalent
to the planar system
\begin{equation}\label{sys2}
\begin{split}
p'&=\sin\theta,\\
 \theta'&=\frac{1}{\sin\theta}-h(p)\cos\theta ,
 \end{split}
\end{equation}
on the phase space $\Theta=I\times(0,\pi)$. We may assume $\theta\in(0,\pi)$, by the symmetry
$(p(s),\theta(s))\mapsto(p(-s),\theta(-s)-\pi)$ of \eqref{sys2}, which
reverses the orientation of $\gamma$. The system has no equilibria in
$\Theta$. We normalise $q(0)=0$ by translation in the $q$-coordinate.

The three functions $h$ are
\begin{equation}\label{eq:h}
h(p)=\frac1p,\qquad h(p)=\cot p,\qquad h(p)=\coth p .
\end{equation}

Since $f(r)=r+O(r^3)$ and $f'(r)=1+O(r^2)$ as $r\to0^+$, in the three cases
\begin{equation}\label{eq:asym}
f(r)\sim r,\qquad h(r)=\frac1r+O(r)\quad (r\to0^+).
\end{equation}

In the three cases $h$ is strictly decreasing on $I$. If $f=p$ or $f=\sinh p$,
then $h>0$ on $I$ and
\begin{equation}\label{eq:hlower}
h(p)\ \ge\ \frac1p\qquad\text{on }I ,
\end{equation}
with equality for $f(p)=p$ and strict inequality for $f(p)=\sinh p$, since
$\coth p>1/p$ for $p>0$.

\begin{definition}\label{def:maximal}
A \emph{maximal soliton} is the curve $\gamma(s)=(q(s),p(s))$ with
$q(s)=q_0+\int_0^s\cos\theta$, where $(p,\theta)\colon(s_m,s_M)\to\Theta$ is a
maximal solution of \eqref{sys2}.
\end{definition}

Equivalently, a maximal soliton is a soliton parametrised by Euclidean arc length
that cannot be extended inside $\Theta$ as a solution of \eqref{sys}. All
solitons below are maximal, and two solitons are identified if they differ by a
translation in $q$, that is, by an isometry of $\h^2$ preserving $X$.

Since $p'=\sin\theta>0$ on $\Theta$, we have:

\begin{lemma}\label{lem:graph}
Along any soliton, $p$ is strictly increasing, hence every soliton is a graph
over the $p$-axis and $\theta$ may be regarded as a function of $p$, with
\begin{equation}\label{eq:dthdp}
\frac{d\theta}{dp}=\frac{1}{\sin^2\theta}-h(p)\cot\theta .
\end{equation}
\end{lemma}

The level sets of $p$ are the horizontal lines in the parabolic case, the rays
from $0$ in the hyperbolic case and the hyperbolic circles centred at $i$ in the
elliptic case, and a soliton meets each of them at most once.

In the elliptic case a soliton is a curve in $M_f$, and we identify it with its
projection to $\h^2\setminus\{i\}$. No soliton of $\h^2$ passes through $i$:
there $X_{\mathrm{ell}}=0$, and \eqref{soliton} would give $1/\kappa=0$. By
Lemma~\ref{lem:graph}, $t$ and hence $\rho$ are strictly monotone along a
soliton, so its projection to $\h^2$ is embedded even when $\nu$ varies by more
than $2\pi$, which happens for $p_0$ close to the critical value $H$ of
Section~\ref{sec:critical}.

\begin{proposition}\label{prop:cc}
Let $\gamma$ be a soliton of the ICSF in $M_f$ with respect to $X=\partial_q$.
Then
\begin{equation}\label{eq:dkappa}
\kappa'=\frac{f'(p)\sin\theta-f(p)\cos\theta}{\sin^3\theta} .
\end{equation}
Moreover, $\kappa$ is constant along $\gamma$ if and only if $\gamma$ lies in
\[
\Sigma_f=\big\{(p,\theta)\in\Theta\colon f'(p)\sin\theta=f(p)\cos\theta\big\},
\]
and then $\kappa\equiv\sqrt{f^2+f'^2}$. A soliton of constant curvature exists
only for $f(p)=\sin p$. In that case $\Sigma_f$ is the diagonal $\{p=\theta\}$
and $\kappa\equiv1$.
\end{proposition}

\begin{proof}
Differentiating \eqref{eq:kap} and using \eqref{sys},
\[
\kappa'=\frac{f'(p)p'}{\sin\theta}-\frac{f(p)\cos\theta}{\sin^2\theta}\,\theta'
=f'(p)+\frac{f'(p)\cos^2\theta}{\sin^2\theta}-\frac{f(p)\cos\theta}{\sin^3\theta},
\]
which is \eqref{eq:dkappa}. Therefore $\kappa'=0$ exactly on $\Sigma_f$. On
$\Sigma_f$ the vectors $(\sin\theta,\cos\theta)$ and $(f,f')$ are proportional,
with positive factor because $f>0$ and $\sin\theta>0$. Hence
$\sin\theta=f/\sqrt{f^2+f'^2}$ and $\kappa=f/\sin\theta=\sqrt{f^2+f'^2}$.

If $\kappa\equiv\kappa_0$ along $\gamma$, then $\gamma$ lies in $\Sigma_f$ and
$\sqrt{f^2+f'^2}=\kappa_0$ at every point $p(s)$. By Lemma~\ref{lem:graph} the set
of such $p$ is an interval, and $f^2+f'^2$ is analytic, so
$f^2+f'^2\equiv\kappa_0^2$ on $I$. Now $f^2+f'^2$ equals $1+p^2$, $1$ and
$\cosh2p$ in the parabolic, hyperbolic and elliptic cases, so it is constant
only for $f=\sin p$. In that case the condition defining $\Sigma_f$ is
$\sin(\theta-p)=0$, that is, $\theta=p$, with $\kappa\equiv1$.
\end{proof}

From \eqref{sys}, the Euclidean curvature $\theta'$ vanishes on the set
\begin{equation}\label{eq:Gamma}
\Gamma=\Big\{(p,\theta)\in\Theta\colon h(p)\sin\theta\cos\theta=1\Big\},
\end{equation}
and $\theta$ increases exactly where $h(p)\sin\theta\cos\theta<1$.

\begin{lemma}\label{lem:Gloc}
\begin{enumerate}[label=\textup{(\arabic*)}]
\item If $f=p$ or $f=\sinh p$, then
  $\Gamma\subset\{0<p\le p^\ast,\ 0<\theta<\pi/2\}$, where $p^\ast=h^{-1}(2)$,
  that is, $p^\ast=\frac12$ and $p^\ast=\artanh\frac12$, respectively.
\item If $f=\sin p$, the parts $\Gamma^-$ and $\Gamma^+$ of $\Gamma$ in
  $\{\theta<\pi/2\}$ and $\{\theta>\pi/2\}$ satisfy
  $\Gamma^-\subset\{p\le p^\ast\}$ and $\Gamma^+\subset\{p\ge\pi-p^\ast\}$, where
  $p^\ast=\arctan\frac12$.
\end{enumerate}
\end{lemma}

\begin{proof}
Let $(p,\theta)\in\Gamma$. Then $h(p)\sin\theta\cos\theta=1$ and
$|\sin\theta\cos\theta|\le\frac12$, so $|h(p)|\ge2$. Since $\sin\theta>0$,
$h(p)$ and $\cos\theta$ have the same sign. If $h(p)>0$, then $\theta<\pi/2$ and
$h(p)\ge2$, which means $p\le p^\ast$ because $h$ is decreasing. This proves (1),
since $h>0$ when $f=p$ or $f=\sinh p$, and the assertion on $\Gamma^-$ in (2). If
$f=\sin p$ and $h(p)<0$, then $\theta>\pi/2$ and $h(p)\le-2$, which means
$p\ge\pi-p^\ast$, since $\cot(\pi-p)=-\cot p$.
\end{proof}

\begin{lemma}\label{lem:cross}
Let $f=p$ or $f=\sinh p$, and let $(p,\theta)$ be a solution of \eqref{sys2}
with $p(0)=a$ and $\theta(0)=\beta\in(\pi/2,\pi)$. Then
\[
p(s)\ \ge\ a\,e^{-(\beta-\pi/2)/a}\ >\ 0
\]
for $s<0$ as long as $\theta(s)>\pi/2$, and $\theta$ reaches the value $\pi/2$
after a backward Euclidean length at most $\beta-\pi/2$.
\end{lemma}

\begin{proof}
On $\{\theta>\pi/2\}$ we have $\cos\theta<0$ and $h>0$, so by \eqref{sys}
$\theta'\ge1/\sin\theta\ge1$. Going backwards, $\theta$ decreases at rate at
least $1$, so it reaches $\pi/2$ after Euclidean length at most $\beta-\pi/2$,
provided the solution still exists. It remains to bound $p$ from below.

By Lemma~\ref{lem:graph} we may use $p$ as independent variable. Put
$u=\theta-\pi/2\in(0,\pi/2)$ and $v=u/p>0$. Since $\sin\theta=\cos u$ and
$\cot\theta=-\tan u$, equation \eqref{eq:dthdp} becomes
\[
\frac{du}{dp}=\frac{1}{\cos^2u}+h(p)\tan u .
\]
Using \eqref{eq:hlower} and $\tan u\ge u$,
\[
\frac{dv}{dp}=\frac1p\frac{du}{dp}-\frac{u}{p^2}
\ \ge\ \frac{1}{p\cos^2 u}+\frac{\tan u-u}{p^{2}}\ \ge\ \frac1p .
\]
Integrating from $p$ to $a$ gives $v(p)\le v(a)-\log(a/p)$. As $v>0$ while
$\theta>\pi/2$, this forces $\log(a/p)<v(a)=(\beta-\pi/2)/a$, which is the stated
bound. In particular $p$ stays in a compact subinterval of $I$ and $\theta$ in
$[\pi/2,\beta]$, so the backward solution does not cease to exist before
$\theta=\pi/2$ is attained.
\end{proof}

\begin{theorem}\label{th:normal}
Let $f=p$ or $f=\sinh p$. Then every maximal soliton meets the set $\{\theta=\pi/2\}$
exactly once. Consequently the family of maximal solitons, up to translations in
$q$, is parametrised by the value $p_0\in I$ of $p$ at that point.
\end{theorem}

\begin{proof}
We first show that a maximal soliton meets $\{\theta=\pi/2\}$ at most once. At a
point with $\theta=\pi/2$ we have $\theta'=1$, so every crossing of
$\{\theta=\pi/2\}$ is transversal and in the direction of increasing $\theta$.
Two crossings would force $\theta$ to return below $\pi/2$, which is impossible.

Now we show that it meets $\{\theta=\pi/2\}$. If $\theta(0)>\pi/2$, this follows
from Lemma~\ref{lem:cross}. Suppose $\theta(0)<\pi/2$, and fix $p_1>p^\ast$. If
the initial point lies outside $\Gamma$, then $\theta'>0$ and $\theta$
increases. As long as $\theta<\pi/2$ we have $p'=\sin\theta\ge\sin\theta(0)>0$,
so $p$ exceeds $p_1$ in finite time unless $\theta$ reaches $\pi/2$ first. Once
$p>p_1$ we have $h(p)<h(p_1)<2$ and $h(p)\sin\theta\cos\theta\le h(p_1)/2<1$,
whence
\[
\theta'\ \ge\ \frac{1-h(p_1)/2}{\sin\theta}\ \ge\ 1-\frac{h(p_1)}{2}\ >\ 0 ,
\]
and $\theta$ reaches $\pi/2$.

If the initial point lies inside $\Gamma$, that is, $h\sin\theta\cos\theta>1$, set
$W=h(p)\sin\theta\cos\theta-1>0$. On $\Gamma$ one has $\theta'=0$ and therefore
\[
W'=h'(p)\,p'\,\sin\theta\cos\theta=h'(p)\sin^2\theta\cos\theta<0 ,
\]
since $h'<0$ and $0<\theta<\pi/2$ on $\Gamma$. Hence $\Gamma$ is crossed only
outwards. The orbit leaves $\{W>0\}$. Otherwise $\theta$ would
decrease to some $\theta_\infty\ge0$ while $p$ increases to some
$p_\infty\in(0,p^\ast]$. If $\theta_\infty>0$, the limit point
$(p_\infty,\theta_\infty)$ lies in $\Theta$, so the solution is defined for all
$s>0$, and $p'=\sin\theta\ge\sin\theta_\infty>0$ would make $p$ unbounded, a
contradiction. Hence $\theta_\infty=0$. Then
$h\sin\theta\cos\theta\to0<1$, contradicting $W>0$. Once outside, the
previous paragraph applies.
\end{proof}

%%%%%%%%%%%%%%%%%%%%%%%%%%%%%%%%%%%%%%%%%%%%%%%%%%%%%%%%%%%%%%%%%%%%%%%%%%%%
\section{The two ends of a soliton}\label{sec:ends}

We normalise maximal solitons by $\theta(0)=\pi/2$, $p(0)=p_0$ and $q(0)=0$. The
end along which $\theta$ increases ($s\to s_M$) is described in
Theorem~\ref{th:upper} in the parabolic and elliptic cases and in
Theorem~\ref{th:hypends} in the hyperbolic case. The end along which $\theta$
decreases ($s\to s_m$) is described in Lemmas~\ref{lem:typeI}
and~\ref{lem:typeII}, which hold in the three cases.

\begin{theorem}\label{th:upper}
Assume the parabolic or the elliptic case, and let $\gamma$ be a maximal soliton. Then $s_M<\infty$ and
\[
\lim_{s\to s_M}\theta(s)=\pi,\qquad
\lim_{s\to s_M}p(s)=p_M\in I,\qquad
\lim_{s\to s_M}q(s)=q_M\in\R ,
\]
with $s_M\le1$ and $|q_M|\le1$. Moreover
\[
\lim_{s\to s_M}\kappa(s)=+\infty ,
\]
the limiting tangent direction is $\partial_q$, the endpoint $(q_M,p_M)$ lies in
the interior of $\h^2$ at finite hyperbolic distance, so $\gamma$ is
incomplete at that end. It is of class $C^{1,1/2}$ but not $C^2$ up to the
endpoint:
\[
p=p_M-C\,|q-q_M|^{3/2}+o\big(|q-q_M|^{3/2}\big),\qquad C=\frac{2\sqrt2}{3} .
\]
\end{theorem}

\begin{proof}
For $\theta\in[\pi/2,\pi)$ we have $\cos\theta\le0$ and $h>0$, so
$\theta'\ge1/\sin\theta>0$ and $\theta$ is increasing on $[0,s_M)$.
Writing $ds=d\theta/\theta'\le\sin\theta\,d\theta$ and integrating,
\[
s_M\le\int_{\pi/2}^{\pi}\sin\theta\,d\theta=1 .
\]
Since $p(0)=p_0$, $q(0)=0$, $|p'|\le1$, $|q'|\le1$ and $s_M\le1$, both $p$ and $q$ have finite limits $p_M$, $q_M$, with $p_M\le p_0+1$ and
$|q_M|\le1$, and $p_M\in I$ because $p_0\le p_M<\infty$ and $I=(0,\infty)$. If
$\theta(s)\to\theta^\ast<\pi$ then
$(q_M,p_M,\theta^\ast)$ would be an interior point of the phase space and the
solution could be continued, contradicting maximality. Hence $\theta\to\pi$.

By \eqref{eq:kap}, $\kappa=f/\sin\theta\to+\infty$ because $f(p_M)>0$. The
hyperbolic length of the end is $\int ds/f(p)\le s_M/\min f<\infty$.

For the asymptotics, near $s_M$ we have $\theta'=\frac{1}{\sin\theta}(1+o(1))$,
so with $\psi=\pi-\theta$ one gets $\psi'=-\frac{1}{\psi}(1+o(1))$, whence
$(\psi^2)'=2\psi\psi'=-2(1+o(1))$. Integrating from $s$ to $s_M$ and using
$\psi(s_M)=0$ gives $\psi(s)^2=2(s_M-s)(1+o(1))$, that is,
$\psi\sim\sqrt{2(s_M-s)}$. Then $\kappa_e=\theta'\sim(2(s_M-s))^{-1/2}$, which
is integrable, and $|q-q_M|=\int_s^{s_M}|\cos\theta|\,d\sigma\sim s_M-s$ while
$p_M-p=\int_s^{s_M}\sin\theta\,d\sigma\sim\tfrac{2\sqrt2}{3}\,(s_M-s)^{3/2}$.
Eliminating $s_M-s$ gives the stated expansion, with $C=2\sqrt2/3$ independent
of the case. Moreover $dp/dq$ is of order $|q-q_M|^{1/2}$ and
$\kappa_e\to+\infty$, which gives the regularity claim.
\end{proof}

\begin{remark}
Setting $h=0$ in \eqref{sys2} gives the Euclidean translating solitons of the
ICSF, the cycloids $q=\tfrac12\sin^2\theta$, $p=\tfrac14(2\theta-\sin2\theta)$ of
radius $1/4$ \cite{dhw,kp}, with a cusp at $\theta=\pi$. Near a singular
endpoint $h(p)\cos\theta$ stays bounded while $1/\sin\theta$ does not, so there
\eqref{sys2} is close to the Euclidean system. This is why $C$ does not depend
on the class.
\end{remark}

Going backwards from $(p_0,\pi/2)$, both $p$ and $\theta$ decrease as long as
the orbit stays outside $\Gamma$.

\begin{definition}\label{def:types}
A maximal soliton is of \emph{Type I} if its backward branch meets $\Gamma$, of
\emph{Type II} if it does not and $\lim_{s\to s_m}p(s)>0$, and of
\emph{critical type} if it does not and $\lim_{s\to s_m}p(s)=0$.
\end{definition}

By Lemma~\ref{lem:typeII} the limit of $p$ exists, so every maximal soliton is of
exactly one of these types.

\begin{lemma}[Type I]\label{lem:typeI}
If the backward branch of the orbit meets $\Gamma$, it does so at exactly one
point $\hat s<0$, it
stays inside $\Gamma$ for $s<\hat s$, and
\[
\lim_{s\to s_m}\big(p(s),\theta(s)\big)=\Big(0,\frac\pi2\Big).
\]
Moreover $s_m>-\infty$ and $q$ has a finite limit $q_m$. The corresponding end
of $\gamma$ reaches the ideal boundary orthogonally at the ideal point
corresponding to $(q_m,0)$, has
infinite hyperbolic length, and $\kappa\to0$ along it.
\end{lemma}

\begin{proof}
The backward branch has $0<\theta<\pi/2$, because the crossings of
$\{\theta=\pi/2\}$ are transversal with $\theta'=1>0$ (proof of
Theorem~\ref{th:normal}). Uniqueness of the crossing follows from the computation
in that proof: on
$\Gamma\cap\{\theta<\pi/2\}$ the function
$W=h\sin\theta\cos\theta-1$ satisfies $W'=h'\sin^2\theta\cos\theta<0$, so
$\Gamma$ is crossed only in the direction of decreasing $W$. Backwards, the
orbit enters $\{W>0\}$ and cannot leave it.

Write $\hat p=p(\hat s)$ and $\hat\theta=\theta(\hat s)$. Inside $\Gamma$ we have $\theta'<0$, so backwards $\theta$ increases and $p$
decreases. For $s<\hat s$ we have $\theta\ge\hat\theta$, so
$p'=\sin\theta\ge\sin\hat\theta>0$ and $s_m>-\infty$. Let $(p_\ast,\theta_\ast)$
be the backward limit, with $\theta_\ast\le\pi/2$. If $p_\ast>0$ the limit is interior and the solution
extends, contradicting maximality, so $p_\ast=0$. Suppose
$\theta_\ast<\pi/2$. Then on $(0,\hat p)$ the function $\theta(p)$ takes values
in a compact subinterval of $(0,\pi/2)$, so there are $A,B>0$ with
$1/\sin^2\theta\le A$ and $\cot\theta\ge B$. By \eqref{eq:dthdp},
\[
\frac{d\theta}{dp}\le A-B\,h(p),
\]
and integrating from $p$ to $\hat p$,
\[
\hat\theta-\theta(p)\le A(\hat p-p)-B\int_p^{\hat p}h(r)\,dr .
\]
By \eqref{eq:asym} the integral tends to $+\infty$ as $p\to0^+$, so the
right-hand side tends to $-\infty$, while the left-hand side tends to
$\hat\theta-\theta_\ast$, a contradiction.
Hence $\theta_\ast=\pi/2$.

Finally $\kappa=f(p)/\sin\theta\to0$ and the tangent becomes orthogonal to
$\partial_q$. Since $dq/dp=\cot\theta\to0$, $q$ has a finite limit $q_m$. As
$dp=\sin\theta\,ds\le ds$, the hyperbolic length of the end
between $p$ and $\hat p$ is at least $\int_p^{\hat p}dr/f(r)$, which tends to
$+\infty$ as $p\to0^+$ by \eqref{eq:asym}.
\end{proof}

\begin{lemma} \label{lem:typeII}
If the backward branch of the orbit does not meet $\Gamma$, then
\[
\lim_{s\to s_m}\theta(s)=0,\qquad \lim_{s\to s_m}p(s)=p_m\ge0 .
\]
If $p_m>0$ (Type II) then $s_m>-\infty$, $q$ has a finite limit $q_m$, and
$\kappa\to+\infty$: this end of the soliton is a singular endpoint, with the same
local behaviour as in Theorem~\ref{th:upper}. If $p_m=0$ (critical type) the orbit converges to the corner $(0,0)$ of
$\Theta$. This end is described in Theorems~\ref{th:class}(2)
and~\ref{th:hyp}(2).
\end{lemma}

\begin{proof}
Outside $\Gamma$ we have $\theta'>0$, so backwards $\theta$ decreases to some
$\theta_\infty\ge0$ and $p$ decreases to some $p_m\ge0$. If
$\theta_\infty>0$, then $p'=\sin\theta\ge\sin\theta_\infty>0$ and the limit is
attained in finite backward length. The limit point must lie on the boundary of
$\Theta$, so $p_m=0$. But then $h(p)\to+\infty$, by \eqref{eq:asym}, and
$h(p)\sin\theta\cos\theta\to+\infty>1$ if $\theta_\infty\in(0,\pi/2)$, so the
orbit would have entered $\Gamma$, a contradiction. Hence $\theta_\infty=0$.

If $p_m>0$ then $h(p)\to h(p_m)<\infty$ and
$ds=\frac{\sin\theta\,d\theta}{1-h(p)\sin\theta\cos\theta}\sim\theta\,d\theta$
as $\theta\to0$, which is integrable, so $s_m>-\infty$ and $q_m$ is finite. By
\eqref{eq:kap}, $\kappa=f(p_m)/\sin\theta\to+\infty$.
\end{proof}

\begin{proposition}\label{prop:corner}
In each of the three cases there is exactly one orbit of \eqref{sys2} whose
backward limit is the corner $(p,\theta)=(0,0)$ of
$\Theta$, that is, with $p(s)\to0$ and $\theta(s)\to0$ as $s\to s_m$. Along it
$p/\theta\to1$, so the corner is approached tangentially to the line $p=\theta$.
For $f=\sin p$ this orbit is the diagonal $\{p=\theta\}$.
\end{proposition}
\begin{proof}
Write $h(p)=\frac1p+c_1p+O(p^3)$, with $c_1=0,\,-\frac13,\,\frac13$ for $1/p$,
$\cot p$, $\coth p$. Set $g=1-h(p)\sin\theta\cos\theta$, so that
$\theta'=g/\sin\theta$. In terms of $d\tau=ds/\sin\theta$, the system
\eqref{sys2} becomes the regular system
\begin{equation}\label{eq:desing}
\frac{dp}{d\tau}=\sin^2\theta,\qquad
\frac{d\theta}{d\tau}=g=1-h(p)\sin\theta\cos\theta .
\end{equation}
Let $\alpha$ be an orbit with $(p,\theta)\to (0,0)$ as $s\to s_m$. In what
follows, $O(\theta^k)$ denotes a function bounded by a constant times $\theta^k$
for $\theta$ small, in a neighbourhood of the point considered.

\emph{(a) $\alpha$ lies outside $\Gamma$ near $s_m$.} As in the proof of
Lemma~\ref{lem:typeI}, going backwards an orbit crosses $\Gamma$ at most once,
and then stays inside, where $\theta'<0$. Backwards, $\theta$ would increase and
could not tend to $0$, so $g>0$ for $s$ close to $s_m$. Both $p$ and $\theta$
increase with $s$ there, and we use $\theta$ as independent variable. Put
$n=\theta/p$, a positive function along the orbit.

\emph{(b) There is $\Lambda>0$ such that $n\le 1+\Lambda\theta^2$ for $\theta$ small.}
There is $C'>0$, depending only on $f$, with
$h(p)\ge1/p-C'p$ and $\sin\theta\cos\theta\ge\theta-C'\theta^3$ for small $p$
and $\theta$. If $n\le\frac12$ there is nothing to prove, so assume $n>\frac12$,
whence $1/n<2$. Using $p=\theta/n$ one gets
\[
h(p)\sin\theta\cos\theta\ \ge\ n-C'\theta^2\Big(n+\frac1n\Big)\ \ge\
n\big(1-C'\theta^2\big)-2C'\theta^2 ,
\]
so $g>0$ forces $n(1-C'\theta^2)<1+2C'\theta^2$ and hence
$n\le1+\Lambda\theta^2$ for $\theta$ small, with $\Lambda=4C'$.

\emph{(c) $n\to1$.} Suppose $n\not\to1$. By (b), there are then
$\varepsilon\in(0,1)$ and a sequence $\theta_k\to0^+$ along the orbit with
$n(\theta_k)\le1-\varepsilon$. Fix one such $\theta_1=\theta_k$, with $k$ to be
chosen large. Along the orbit,
\[
\frac{dn}{d\theta}=\frac{n\,(g-n\sin^2\theta)}{\theta\,g},
\]
and as long as $n\le1-\varepsilon/2$ and $\theta$ is small we have
$g\ge\varepsilon/4>2\sin^2\theta$, hence $dn/d\theta>0$. Therefore $n$ stays
below $1-\varepsilon/2$ for all $\theta<\theta_1$. Then $g\ge\varepsilon/4$ on
$(0,\theta_1)$ and $dp/d\theta=\sin^2\theta/g\le(4/\varepsilon)\theta^2$, so
$p(\theta_1)\le\frac{4}{3\varepsilon}\theta_1^3$ because $p\to0$, whence
\[
n(\theta_1)=\frac{\theta_1}{p(\theta_1)}\ \ge\ \frac{3\varepsilon}{4}\,
\theta_1^{-2}.
\]
The smallness conditions above depend only on $\varepsilon$, so they hold for
every $\theta_k$ with $k$ large. Choosing $k$ so large that
$\frac{3\varepsilon}{4}\theta_k^{-2}>1$ contradicts $n(\theta_k)\le1-\varepsilon$.
Hence $n\to1$.

\emph{(d) Blow-up.} Put $d\varsigma=d\tau/\theta$. Then
$d\theta/d\varsigma=\theta g$ and $dn/d\varsigma=n(g-n\sin^2\theta)$, and
expanding $h\sin\theta\cos\theta=(\frac1p+c_1p+\dots)
(\theta-\frac23\theta^3+\dots)$ with $p\theta=\theta^2/n$ gives
$g=(1-n)+(\frac{2n}3-\frac{c_1}{n})\theta^2+O(\theta^4)$. Near
$(n,\theta)=(1,0)$,
\begin{equation}\label{eq:blowup}
\begin{aligned}
\frac{dn}{d\varsigma}&=n(1-n)-\Big(\frac{n^2}{3}+c_1\Big)\theta^2+O(\theta^4),\\
\frac{d\theta}{d\varsigma}&=\theta(1-n)+\Big(\frac{2n}{3}-\frac{c_1}{n}\Big)\theta^3+O(\theta^5).
\end{aligned}
\end{equation}
By (b) and (c) the orbit $\alpha$ stays in a compact neighbourhood of $(1,0)$,
so it exists for all $\varsigma\le0$ and converges to $(1,0)$ as
$\varsigma\to-\infty$.

\emph{(e) The flow near $(1,0)$.} The linearisation at $(1,0)$ is
$\mathrm{diag}(-1,0)$. By the centre manifold theorem \cite{ca} there is a
locally invariant curve $n=\mu(\theta)$ through $(1,0)$, tangent to the centre
direction. Writing $\mu(\theta)=1+a_1\theta+a_2\theta^2+O(\theta^3)$ and
inserting it into the invariance equation
$\mu'(\theta)\,\frac{d\theta}{d\varsigma}=\frac{dn}{d\varsigma}$, the
right-hand side is $-a_1\theta+O(\theta^2)$ while the left-hand side is
$O(\theta^2)$, so $a_1=0$. At the next order the right-hand side is
$-(a_2+\frac13+c_1)\theta^2+O(\theta^4)$ and the left-hand side is
$O(\theta^4)$, so, writing $a=a_2$,
\[
a=-\Big(\tfrac13+c_1\Big)=-\tfrac13,\ 0,\ -\tfrac23
\quad\text{for } f=p,\ \sin p,\ \sinh p .
\]
The flow on the invariant curve is
\[
\frac{d\theta}{d\varsigma}=\theta^3\Big(-a+\frac23-c_1\Big)+O(\theta^5)
=\theta^3+O(\theta^5),
\]
the same in the three cases, and it is positive for $\theta>0$.

\emph{(f) Existence and uniqueness.} Shrink the neighbourhood of $(1,0)$ so
that \eqref{eq:blowup} and the expansions above hold on it. On the branch
$\theta>0$ of the invariant curve, $d\theta/d\varsigma>0$, so going backwards
$\theta$ decreases to $0$ and $n=\mu(\theta)\to1$. That branch is therefore a
single orbit converging to $(1,0)$ as $\varsigma\to-\infty$, which gives
existence.

For uniqueness, put $u=n-\mu(\theta)$. Since the curve $\{u=0\}$ is invariant
and the linearisation in the $n$-direction is $-1$,
\[
\frac{du}{d\varsigma}=-u+O\big((|u|+\theta)\,|u|\big),
\]
so, after shrinking the neighbourhood once more,
$\mathrm{sign}(u)\,\frac{du}{d\varsigma}\le-\frac{|u|}{2}$ there. Therefore
$|u(\varsigma)|\ge|u(\varsigma_0)|\,e^{(\varsigma_0-\varsigma)/2}$ for
$\varsigma\le\varsigma_0$ as long as the orbit remains in the neighbourhood, so
an orbit with $u(\varsigma_0)\neq0$ leaves it backwards in finite time. An orbit
converging to $(1,0)$ as $\varsigma\to-\infty$ stays in the neighbourhood, so it
satisfies $u\equiv0$ and coincides with the branch above. By (c) every orbit of
\eqref{sys2} converging to the corner enters this neighbourhood with $n\to1$, so
it is that orbit.

For $f=\sin p$ the diagonal is an orbit (Theorem~\ref{th:horo}) which converges
to the corner, so by uniqueness it is the orbit of the statement.
\end{proof}

%%%%%%%%%%%%%%%%%%%%%%%%%%%%%%%%%%%%%%%%%%%%%%%%%%%%%%%%%%%%%%%%%%%%%%%%%%%%
\begin{figure}[ht]
\centering
\includegraphics[width=\textwidth]{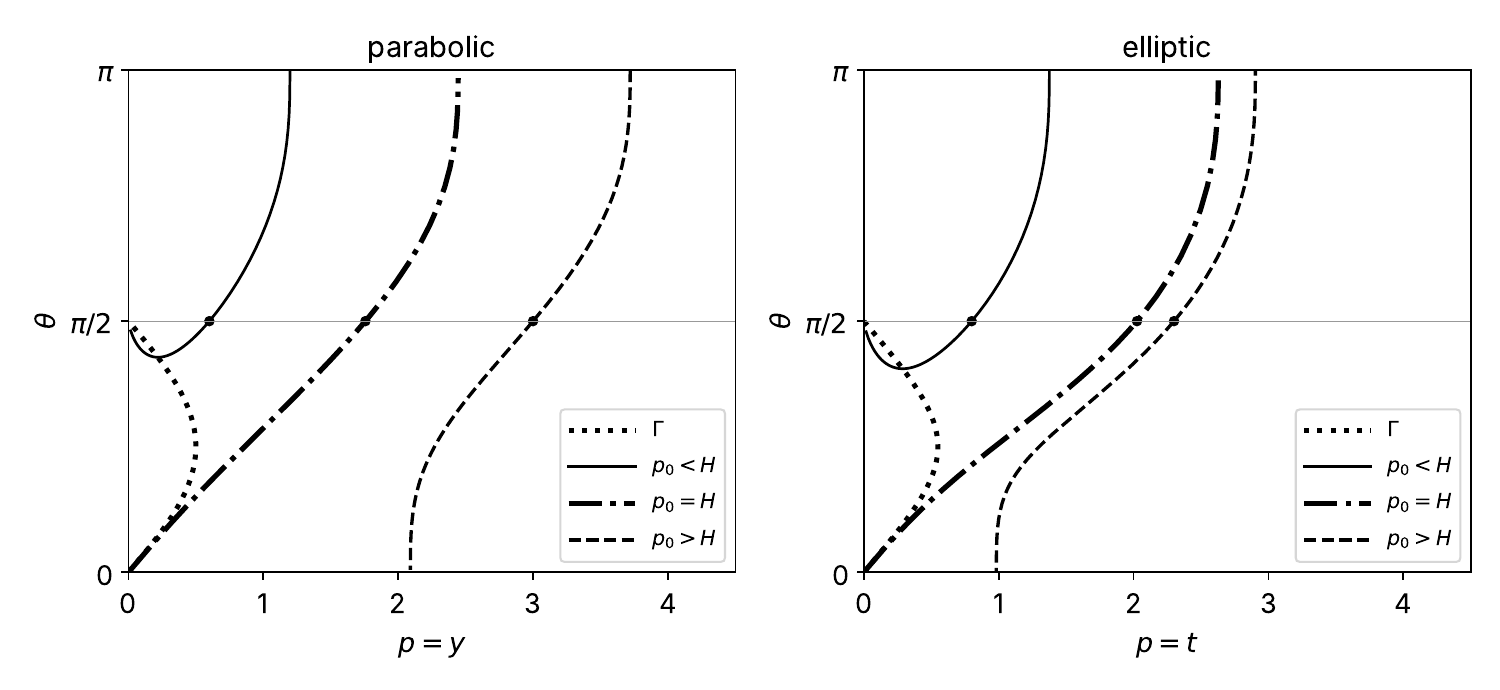}
\caption{Phase portraits of \eqref{sys2} in the $(p,\theta)$-plane for the
parabolic case ($p=y$) and the elliptic case ($p=t$). The dotted curve is $\Gamma$ and the dots mark the
points with $\theta=\pi/2$. Solid lines are orbits of Type~I and dashed lines
orbits of Type~II. The dash-dotted line is the orbit of critical type, which
leaves the corner $(0,0)$ tangentially to the line $p=\theta$
(Proposition~\ref{prop:corner}).}
\label{fig:phase}
\end{figure}

\section{The critical soliton}\label{sec:critical}

In this section $f=p$ or $f=\sinh p$, so that \eqref{eq:hlower} holds, and
$\alpha_{p_0}$ denotes the orbit of \eqref{sys2} through $(p_0,\pi/2)$. The
hyperbolic case is treated in Section~\ref{sec:classes}.

\begin{lemma}\label{lem:order}
If $0<p_0<p_1$ and $\alpha_{p_1}$ is of Type I, then so is $\alpha_{p_0}$.
\end{lemma}

\begin{proof}
Let $\hat\theta\in(0,\pi/2)$ be the value of $\theta$ at which $\alpha_{p_1}$
meets $\Gamma$. If $\alpha_{p_0}$ meets $\Gamma$ at some $\theta>\hat\theta$, it
is of Type I. Otherwise it stays outside $\Gamma$ on $[\hat\theta,\pi/2]$. Outside $\Gamma$ both $p$ and $\theta$ increase
with $s$, so $p$ is a strictly increasing function of $\theta$. The backward
branch of $\alpha_{p_0}$ reaches $\hat\theta$: going backwards $\theta$
decreases, and the only way the orbit could cease to exist first is $p\to0$,
which is impossible, because $p\to0$ with $\theta\ge\hat\theta>0$ gives
$h(p)\sin\theta\cos\theta\to+\infty>1$ and the orbit would have entered
$\Gamma$. Both orbits are therefore graphs $p=p(\theta;p_i)$ on
$[\hat\theta,\pi/2]$. They do not meet, so $p(\theta;p_0)<p(\theta;p_1)$ there.
At $\hat\theta$ one has
$h(p(\hat\theta;p_1))\sin\hat\theta\cos\hat\theta=1$. Since
$p(\hat\theta;p_0)<p(\hat\theta;p_1)$ and $h$ is decreasing,
$h(p(\hat\theta;p_0))\sin\hat\theta\cos\hat\theta>1$. So $\alpha_{p_0}$ is
inside $\Gamma$ at $\hat\theta$, a contradiction.
\end{proof}

\begin{lemma}\label{lem:lower}
If $\alpha_{p_0}$ is not of Type I, then $p_0>p^\ast=h^{-1}(2)$.
\end{lemma}

\begin{proof}
If the orbit never meets $\Gamma$, then at $\theta=\pi/4$,
$h(p(\pi/4))\cdot\frac12<1$, i.e.\ $h(p(\pi/4))<2$, so $p(\pi/4)>p^\ast$. As $p$
increases with $\theta$, $p_0=p(\pi/2)>p(\pi/4)>p^\ast$.
\end{proof}

\begin{lemma}\label{lem:upper}
Let $p_1\in I$ with $h(p_1)<2$ and set
$L(p_1)=\dfrac{\pi/4}{1-h(p_1)/2}$. If $p_0\ge p_1+L(p_1)$, then $\alpha_{p_0}$
is of Type II.
\end{lemma}

\begin{proof}
On $\{p\ge p_1\}$ we have $h(p)\le h(p_1)$ and
$1-h(p)\sin\theta\cos\theta\ge1-h(p_1)/2>0$, so by \eqref{eq:dthdp}
\[
\frac{dp}{d\theta}=\frac{\sin^2\theta}{1-h(p)\sin\theta\cos\theta}
\le\frac{\sin^2\theta}{1-h(p_1)/2}.
\]
Letting $\theta$ run from $\pi/2$ down to $0$, the total decrease of $p$ is at
most $\frac{1}{1-h(p_1)/2}\int_0^{\pi/2}\sin^2\theta\,d\theta=L(p_1)$. By
continuity $p$ never reaches $p_1$, so $p\ge p_1>p^\ast$ throughout, and
$\Gamma\subset\{p\le p^\ast\}$ is never met.
\end{proof}

\begin{lemma}\label{lem:trap}
Let $\delta>0$ and choose $\theta_1\in(0,\pi/2)$ with $h(\delta)\,\theta_1\le\frac12$
and $\frac23\theta_1^3<\delta$. If the backward orbit of $\alpha_{p_0}$ reaches a
point $(\tilde p,\theta_1)$ with $\tilde p\ge2\delta$ while staying outside $\Gamma$, then
it stays outside $\Gamma$ thereafter and
\[
p_m=\lim_{s\to s_m}p(s)\ \ge\ \tilde p-\tfrac23\theta_1^3\ >\ \delta\ >\ 0 ,
\]
so $\alpha_{p_0}$ is of Type II.
\end{lemma}

\begin{proof}
Going backwards from $(\tilde p,\theta_1)$, $\theta$ decreases and so does $p$. As
long as $p\ge\delta$ and $\theta\le\theta_1$ we have
$h(p)\sin\theta\cos\theta\le h(\delta)\theta_1\le\frac12$, so the orbit is
outside $\Gamma$ and
\[
\frac{dp}{d\theta}=\frac{\sin^2\theta}{1-h(p)\sin\theta\cos\theta}
\le2\sin^2\theta\le2\theta^2 .
\]
Hence, while these conditions hold, the total decrease of $p$ when $\theta$ goes
from $\theta_1$ to $\theta$ is at most $\frac23\theta_1^3<\delta\le\tilde p-\delta$,
so $p>\tilde p-\delta\ge\delta$. By continuity the condition $p\ge\delta$ is never
violated, and the estimate holds down to $\theta=0$. Lemma~\ref{lem:typeII}
then gives $\theta\to0$ with $p_m\ge\tilde p-\frac23\theta_1^3>\delta$.
\end{proof}

\begin{theorem}\label{th:class}
Assume the parabolic or the elliptic case. There exists $H\in I$ with $H>p^\ast$ and
\[
H\ \le\ p_1+\frac{\pi/4}{1-h(p_1)/2}\qquad\text{for every }p_1\in I
\text{ with }h(p_1)<2,
\]
such that the maximal soliton $\gamma_{p_0}$ satisfies:
\begin{enumerate}[label=\textup{(\arabic*)}]
\item if $0<p_0<H$, then $\gamma_{p_0}$ is of Type I: one end is the singular
endpoint of Theorem~\ref{th:upper}, and the other end reaches the ideal boundary
of $\h^2$ orthogonally, with infinite length and $\kappa\to0$;
\item if $p_0=H$, then $\gamma_H$ is of critical type: its backward orbit
converges to the corner $(p,\theta)=(0,0)$ of $\Theta$, tangentially to the line
$p=\theta$. Along this end $\kappa\to1$, $q\to-\infty$ and $p\sim Ke^{q}$ for
some $K>0$. In the
parabolic case the end converges to the ideal point $\infty$, the fixed point of
the group. In the elliptic case it winds infinitely many times around $i$ and
accumulates on the whole ideal boundary;
\item if $p_0>H$, then $\gamma_{p_0}$ is of Type II: it is an arc of finite
length with compact closure in $\h^2$ and two singular endpoints, $\kappa\to+\infty$
at both, and it is a graph over a bounded interval $(p_m,p_M)$ of the $p$-axis,
concave in the Euclidean sense in the coordinates $(q,p)$.
\end{enumerate}
\end{theorem}

\begin{proof}
By Lemma~\ref{lem:order} the set
$A=\{p_0\colon\text{Type I}\}$ is an interval of the form $(0,H)$ or $(0,H]$,
and by Lemmas~\ref{lem:lower} and \ref{lem:upper} it is non-empty and $H$ is
finite, with the stated bounds. Type I is an open condition, because the
crossing of $\Gamma$ is transversal ($W'<0$ on $\Gamma$) and solutions depend
continuously on initial data. If $\alpha_{p_0}$ meets $\Gamma$ at
$\hat s<0$, then $W$ increases through $0$ there along the backward branch, so
$W(\hat s-\varepsilon)>0$ for small $\varepsilon>0$, and the same holds at the
parameter $\hat s-\varepsilon$ for every orbit with initial value close to
$p_0$. So $A=(0,H)$ and $\alpha_H$ is not of Type I. Since $p^\ast\in A$ by
Lemma~\ref{lem:lower}, $H>p^\ast$.

Type II is open as well. If $\alpha_{p_0}$ is of Type II with limit
$p_m>0$, apply Lemma~\ref{lem:trap} with $\delta=p_m/3$ and the
corresponding $\theta_1$. The backward orbit reaches $\{\theta=\theta_1\}$ at a
point with $p\ge p_m=3\delta>2\delta$, having stayed strictly outside
$\Gamma$ on the compact piece of orbit between $\theta=\pi/2$ and
$\theta=\theta_1$. On that piece $\theta'$ is bounded below by a positive
constant, so the piece is traversed in a bounded parameter interval and the two
strict inequalities, $W<0$ on it and $p>2\delta$ at $\theta=\theta_1$, persist
for all orbits with initial value close to $p_0$. Lemma~\ref{lem:trap}, applied
to them with the same $\delta$ and $\theta_1$, shows that they are of Type II.
Since the points of $(0,H)$ are of Type I, the orbit $\alpha_H$ is not of
Type~II either. By Lemma~\ref{lem:typeII} and Definition~\ref{def:types} it is of
critical type, and by Proposition~\ref{prop:corner} it is the only one: its
backward limit is the corner $(0,0)$, approached tangentially to $p=\theta$.
Consequently every $p_0>H$, being of neither Type I nor of critical type, is of
Type II.

By the proof of Proposition~\ref{prop:corner}, along the critical orbit
$\theta/p=\mu(\theta)=1+O(\theta^2)$, so $\theta=p+O(p^3)$ and
\[
\frac{dq}{dp}=\cot\theta=\frac1p+O(p).
\]
Integrating, $q=\log p+K_0+o(1)$ as $p\to0^+$, that is, $p\sim Ke^{q}$ with
$K=e^{-K_0}$, and $q\to-\infty$. In the parabolic case $(q,p)=(x,y)$, so
$y\sim Ke^{x}$ as $x\to-\infty$ and the end converges to $\infty$. In the
elliptic case $(q,p)=(\nu,t)$, so $\nu\to-\infty$ while $\rho\to\infty$.
Hence every angle about $i$ is attained with $\rho$ arbitrarily large, and the
end accumulates on the whole ideal boundary. Finally, by \eqref{eq:kap} and
\eqref{eq:asym},
\[
\kappa=\frac{f(p)}{\sin\theta}=\frac{p+O(p^3)}{p+O(p^3)}=1+O(p^2)\ \longrightarrow\ 1 .
\]

Statements (1) and (3) follow from Lemmas~\ref{lem:typeI} and
\ref{lem:typeII}. Concavity in (3), in the Euclidean sense in the $(q,p)$-plane:
writing the soliton as $q=q(p)$, we have
\[
\frac{d^2q}{dp^2}=-\frac{\theta'}{\sin^3\theta}
=-\frac{1-h(p)\sin\theta\cos\theta}{\sin^4\theta}<0
\]
outside $\Gamma$.
\end{proof}

\begin{remark}
The critical soliton is the only soliton, in each class, along which $q$ is
unbounded. Along a Type~I end $dq/dp=\cot\theta\to0$, so $q_m$ is finite
(Lemma~\ref{lem:typeI}), and along a singular endpoint the Euclidean length is
finite. Along the critical end $\cot\theta\sim1/p$, which is not integrable at
$p=0$.
\end{remark}

%%%%%%%%%%%%%%%%%%%%%%%%%%%%%%%%%%%%%%%%%%%%%%%%%%%%%%%%%%%%%%%%%%%%%%%%%%%%
\begin{figure}[ht]
\centering
\includegraphics[width=\textwidth]{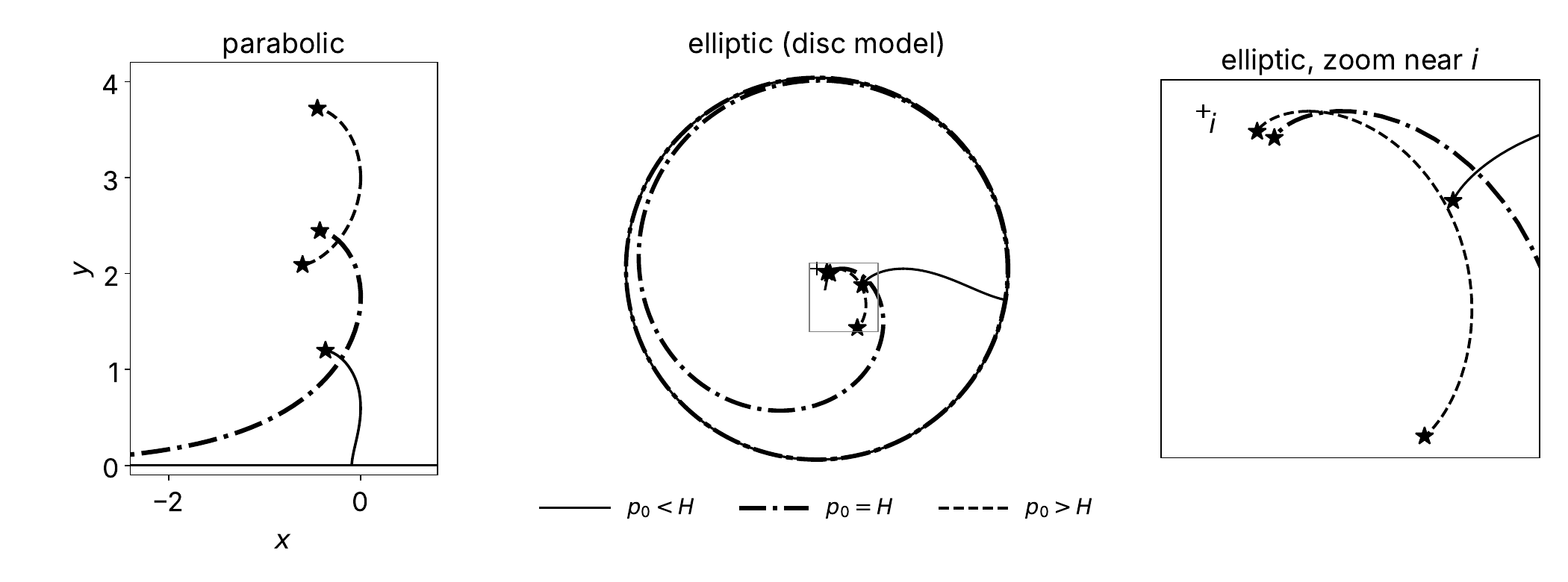}
\caption{The solitons of Figure~\ref{fig:phase}, in the upper half-plane
(parabolic) and in the disc model $w=e^{-t+i\nu}$, centred at $i$ (elliptic).
Stars mark the singular endpoints, where $\kappa\to+\infty$. The solitons of
Type~I (solid) meet the ideal boundary orthogonally. Those of Type~II (dashed)
are arcs of finite length with two singular endpoints. The critical soliton
(dash-dotted) follows $y\sim Ke^{x}$ as $x\to-\infty$ in the parabolic case, and
spirals towards the ideal boundary in the elliptic case. The right panel
enlarges the region near $i$ marked in the disc.}
\label{fig:solitons}
\end{figure}

\section{The three Killing classes}\label{sec:classes}

In the parabolic case $f(p)=p$, $h(p)=1/p$, $I=(0,\infty)$, $(q,p)=(x,y)$ and $X=\partial_x$. The
system \eqref{sys} is
\begin{equation}\label{para}
x'=\cos\theta,\qquad y'=\sin\theta,\qquad
\theta'=\frac{1}{\sin\theta}-\frac{\cos\theta}{y},
\end{equation}
and $\Gamma$ is the curve $y=\frac12\sin2\theta$, $\theta\in(0,\pi/2)$. In the
elliptic case $f(p)=\sinh p$, $h(p)=\coth p$, $I=(0,\infty)$, $(q,p)=(\nu,t)$,
where $\rho=2\artanh(e^{-t})$ is the geodesic distance to $i$, and
$X=X_{\mathrm{ell}}$. The planar system is
\begin{equation}\label{ellsys}
t'=\sin\theta,\qquad
\theta'=\frac{1}{\sin\theta}-\coth t\,\cos\theta .
\end{equation}
In both cases \eqref{eq:hlower} holds, so the results of
Sections~\ref{sec:unified}--\ref{sec:critical} apply.

\begin{theorem}\label{th:parell}
Every maximal parabolic or elliptic soliton is a graph over the $p$-axis, meets
$\{\theta=\pi/2\}$ exactly once, and satisfies the conclusions of
Theorem~\ref{th:class}, with
\begin{enumerate}[label=\textup{(\arabic*)}]
\item $\frac12<H\le\frac12\big(1+\sqrt{\pi/2}\big)^2$ in the parabolic case;
\item $\artanh\frac12<H\le1.22+\dfrac{\pi/4}{1-\frac12\coth1.22}<3.162$ in the
elliptic case.
\end{enumerate}
\end{theorem}

\begin{proof}
The first assertions are Theorems~\ref{th:normal}, \ref{th:upper} and
\ref{th:class}. The lower bounds are $p^\ast=h^{-1}(2)$, given in
Lemma~\ref{lem:Gloc}. The upper bounds follow from Theorem~\ref{th:class}. In the parabolic case the
bound is $p_1+\frac\pi2\,\frac{p_1}{2p_1-1}$, $p_1>\frac12$, whose minimum,
attained at $p_1=\frac12\big(1+\sqrt{\pi/2}\big)$, is
$\frac12\big(1+\sqrt{\pi/2}\big)^2$. In the elliptic case we take $p_1=1.22$, which nearly minimises
$p_1+\frac{\pi/4}{1-\frac12\coth p_1}$.
\end{proof}

Numerically, $H\approx1.75694$ in the parabolic case and $H\approx2.02556$ in the
elliptic case, computed by bisection on $p_0$ using the dichotomy of
Theorem~\ref{th:class}.

In the hyperbolic case $f(p)=\sin p$, $h(p)=\cot p$, $I=(0,\pi)$, $(q,p)=(u,\varphi)$ with
$u=\log r$ and $\varphi$ the polar angle, and $X=x\partial_x+y\partial_y$. Now
$h$ changes sign, and the planar system is
\begin{equation}\label{hypsys}
\varphi'=\sin\theta,\qquad
\theta'=\frac{1}{\sin\theta}-\cot\varphi\,\cos\theta .
\end{equation}
Both $\varphi\to0$ and $\varphi\to\pi$ correspond to the ideal boundary.

\begin{lemma}\label{lem:sym}
The map $S(\varphi,\theta)=(\pi-\varphi,\pi-\theta)$, together with $s\mapsto-s$,
is a symmetry of \eqref{hypsys}. It is induced by the isometry $(x,y)\mapsto
(-x,y)$ of $\h^2$, which preserves $X$. It interchanges the two branches
$\Gamma^\mp$ of Lemma~\ref{lem:Gloc}, and its unique fixed point on
$\{\theta=\pi/2\}$ is $\varphi=\pi/2$.
\end{lemma}

\begin{proof}
Direct substitution: with $\tilde\varphi=\pi-\varphi$, $\tilde\theta=\pi-\theta$
and $\tilde s=-s$ one has $d\tilde\varphi/d\tilde s=\sin\tilde\theta$ and
\[
\frac{d\tilde\theta}{d\tilde s}=\theta'
=\frac{1}{\sin\theta}-\cot\varphi\cos\theta
=\frac{1}{\sin\tilde\theta}-\cot\tilde\varphi\cos\tilde\theta,
\]
since $\cot(\pi-\varphi)=-\cot\varphi$ and $\cos(\pi-\theta)=-\cos\theta$.
\end{proof}

\begin{theorem}\label{th:horo}
The diagonal $\Delta=\{\varphi=\theta\}$ is an orbit of \eqref{hypsys}. The
corresponding solitons are exactly the horocycles
\[
x^2+\Big(y-\frac{1}{2b}\Big)^2=\frac{1}{4b^2},\qquad b>0,
\]
tangent to the ideal point $0$, and they have $\kappa\equiv1$.
\end{theorem}

\begin{proof}
In the regular form \eqref{eq:desing} the field is
$\big(\sin^2\theta,\,1-\cot\varphi\sin\theta\cos\theta\big)$. On
$\varphi=\theta$,
\[
1-\cot\theta\,\sin\theta\cos\theta=1-\cos^2\theta=\sin^2\theta ,
\]
so the field is tangent to $\Delta$ and $\Delta$ is invariant.

On $\Delta$, \eqref{sys} reads $u'=\cos\varphi$, $\varphi'=\sin\varphi$, so
$d\varphi/du=\tan\varphi$ and $\log\sin\varphi=u+\mathrm{const}$, i.e.\
$\sin\varphi=b e^{u}=br$. Since $y=r\sin\varphi$ and $r^2=x^2+y^2$, this is
$y=b(x^2+y^2)$, the stated circle, and
$\kappa=f(\varphi)/\sin\theta=\sin\varphi/\sin\theta=1$ on $\Delta$.
\end{proof}

\begin{lemma}\label{lem:hypnormal}
Every maximal solution of \eqref{hypsys} meets $\{\theta=\pi/2\}$ exactly once.
\end{lemma}

\begin{proof}
At a point with $\theta=\pi/2$ one has $\theta'=1>0$ whatever the sign of
$h$, so the crossing is transversal and in the direction of increasing $\theta$:
there is at most one. For existence we use the diagonal $\Delta=\{\varphi=\theta\}$
of Theorem~\ref{th:horo}, which is an orbit. By uniqueness no other orbit
meets it, so every orbit lies in $\{\varphi<\theta\}$ or in
$\{\varphi>\theta\}$. The symmetry $S$ of Lemma~\ref{lem:sym} exchanges these two
half-planes and leaves $\{\theta=\pi/2\}$ invariant, so it suffices to treat an
orbit in $B=\{\varphi<\theta\}$.

\emph{Case $\theta<\pi/2$.} Then $\varphi<\theta<\pi/2$, so $h=\cot\varphi$ is
positive and decreasing on the range of $\varphi$, and the proof of
Theorem~\ref{th:normal} for $\theta<\pi/2$ applies verbatim. It uses only that
$h$ is positive and decreasing, and that $\varphi$ stays below $\pi/2$ until
$\theta$ reaches $\pi/2$.

\emph{Case $\theta>\pi/2$.} We claim $\theta'>0$ throughout
$B\cap\{\theta>\pi/2\}$. If $\varphi\le\pi/2$ then $h\ge0$ and $\sin\theta\cos\theta<0$,
so $1-h\sin\theta\cos\theta>0$. If $\pi/2<\varphi<\theta<\pi$ then
$|\cot\varphi|<|\cot\theta|$ because $|\cot|$ increases on $(\pi/2,\pi)$, whence
\[
h\sin\theta\cos\theta=|\cot\varphi|\,|\sin\theta\cos\theta|
<|\cot\theta|\,|\sin\theta\cos\theta|=\cos^2\theta<1 .
\]
Moreover $1-h\sin\theta\cos\theta\ge\sin^2\theta$ in both cases, so going
backwards $\theta$ decreases, at rate at least
$\sin\theta\ge\sin\theta(0)>0$. Along $[\pi/2,\theta(0)]$ the angle $\varphi$ stays
below $\pi$, so the only way for the backward orbit to cease to exist before
$\theta$ reaches $\pi/2$ is $\varphi\to0$. We exclude this. Let $\theta^\ast\ge\pi/2$
be the backward limit of $\theta$.

\emph{If $\theta^\ast>\pi/2$}: for $\varphi$ small we have $h=\cot\varphi>0$,
$\tan\varphi/\varphi\le2$ and $1-h\sin\theta\cos\theta\ge h\,|\sin\theta\cos\theta|$,
while $|\cos\theta|\ge|\cos\theta^\ast|>0$ on $[\theta^\ast,\theta(0)]$. Therefore
\[
\frac{d\log\varphi}{d\theta}
=\frac{\sin^2\theta}{\varphi\,(1-h\sin\theta\cos\theta)}
\le\frac{\tan\varphi}{\varphi}\,\frac{\sin\theta}{|\cos\theta|},
\]
which is bounded on $[\theta^\ast,\theta(0)]$ since $\theta^\ast>\pi/2$. So
$\log\varphi$ stays bounded and $\varphi\not\to0$.

\emph{If $\theta^\ast=\pi/2$}: put $u=\theta-\pi/2>0$ and use $\varphi$ as
independent variable. Both $u$ and $\varphi$ tend to $0$ backwards, so we may
fix $\varphi_0>0$ so small that $\varphi u\le\frac12$ on $(0,\varphi_0]$ and
that the two elementary inequalities used below hold there. Then
\[
\frac{du}{d\varphi}=\frac{1+\cot\varphi\,\sin\theta\,|\cos\theta|}{\sin^2\theta}
\ \ge\ 1+\cot\varphi\,\tan u\ \ge\ 1+\Big(\frac1\varphi-\varphi\Big)u ,
\]
using $\sin\theta|\cos\theta|/\sin^2\theta=\tan u\ge u$ and $\cot\varphi\ge1/\varphi-\varphi$
for small $\varphi$. Therefore
\[
\frac{d}{d\varphi}\Big(\frac u\varphi\Big)=\frac1\varphi\Big(\frac{du}{d\varphi}-\frac u\varphi\Big)
\ \ge\ \frac{1-\varphi u}{\varphi}\ \ge\ \frac1{2\varphi}
\]
on $(0,\varphi_0]$. Integrating from $\varphi$ to $\varphi_0$,
\[
\frac{u(\varphi_0)}{\varphi_0}-\frac{u(\varphi)}{\varphi}\ \ge\
\frac12\log\frac{\varphi_0}{\varphi}\ \xrightarrow[\varphi\to0^+]{}\ +\infty ,
\]
so $u/\varphi\to-\infty$, which contradicts $u>0$.

Hence the orbit reaches $\theta=\pi/2$.
\end{proof}

\begin{theorem}\label{th:hyp}
For hyperbolic solitons the critical value is $H=\pi/2$, and the critical soliton is
the horocycle of Theorem~\ref{th:horo}. More precisely, for an orbit through
$(\varphi_0,\pi/2)$:
\begin{enumerate}[label=\textup{(\arabic*)}]
\item if $\varphi_0<\pi/2$ the orbit meets $\Gamma^-$ and the corresponding end
reaches the ideal boundary orthogonally;
\item if $\varphi_0=\pi/2$ the soliton is the horocycle $y=b(x^2+y^2)$, whose two
ends converge to the ideal point $0$, fixed by the group, with $u\to-\infty$
along both;
\item if $\varphi_0>\pi/2$ the orbit is the image under $S$ of an orbit of case
\textup{(1)}, and the soliton is the reflection in the $y$-axis of a soliton of
case \textup{(1)}.
\end{enumerate}
\end{theorem}

\begin{proof}
On $\Delta$ one has $\cot\varphi\sin\theta\cos\theta=\cos^2\theta\le1$, with
equality only at $\theta=0$. Since $h=\cot$ is decreasing,
$\cot\varphi\sin\theta\cos\theta>1$ requires $\varphi<\theta$, hence
$\Gamma^-\subset\{\varphi<\theta\}$, strictly below $\Delta$. By
Lemma~\ref{lem:sym}, $\Gamma^+=S(\Gamma^-)\subset\{\varphi>\theta\}$, strictly
above $\Delta$, so an orbit below $\Delta$ never meets $\Gamma^+$.

Orbits cannot cross $\Delta$, which is itself an orbit. If $\varphi_0>\pi/2$ the
orbit lies above $\Delta$ and therefore never meets $\Gamma^-$. If
$\varphi_0<\pi/2$ the orbit lies below $\Delta$, so $\varphi<\theta$ along it.
In particular its backward branch stays in $\{\varphi<\theta<\pi/2\}$, where
$h=\cot\varphi>0$, so Lemma~\ref{lem:typeII} applies. Were the orbit not to meet
$\Gamma^-$, that lemma would give $\theta\to0$, hence $\varphi\to0$ because
$0<\varphi<\theta$. The orbit would then converge to the corner, and by
Proposition~\ref{prop:corner} it would be $\Delta$, which is impossible since it
lies strictly below $\Delta$.
Hence it meets $\Gamma^-$, and Lemma~\ref{lem:typeI} applies. This proves
$H=\pi/2$ and items (1), (2). Item (3) is Lemma~\ref{lem:sym}.
\end{proof}

\begin{theorem}\label{th:hypends}
Let $\gamma$ be a maximal hyperbolic soliton through $(\varphi_0,\pi/2)$ with
$\varphi_0<\pi/2$. The case $\varphi_0>\pi/2$ follows from Lemma~\ref{lem:sym}.
Then:
\begin{enumerate}[label=\textup{(\arabic*)}]
\item the end along which $\theta$ decreases satisfies
$\varphi\to0$ and $\theta\to\pi/2$: it reaches the ideal boundary orthogonally,
with infinite hyperbolic length and $\kappa\to0$;
\item the end along which $\theta$ increases is a singular endpoint:
$\theta\to\pi$, $\varphi\to\varphi_M\in(0,\pi)$, $u\to u_M\in\R$ and $\kappa\to+\infty$.
The endpoint lies in the interior of $\h^2$ at finite hyperbolic distance, and
$\gamma$ has there the expansion of Theorem~\ref{th:upper} with the same
constant $C=2\sqrt2/3$.
\end{enumerate}
\end{theorem}

\begin{proof}
(1) was proved in Theorem~\ref{th:hyp}.

(2) Let $\alpha$ be the orbit of $\gamma$ and $S$ the symmetry of
Lemma~\ref{lem:sym}. The forward branch of $\alpha$, on which $\theta\in(\pi/2,\pi)$, is mapped by $S$, with $s\mapsto-s$,
onto the backward branch of the orbit $\alpha'$ through
$(\pi-\varphi_0,\pi/2)$, on which $\theta\in(0,\pi/2)$. Since $\varphi_0<\pi/2$, the orbit $\alpha$ lies below $\Delta$, so $\alpha'$ lies
above $\Delta$, because $S$ interchanges the two sides of $\Delta$. Also
$\Gamma^-$ lies strictly below $\Delta$ by the proof of Theorem~\ref{th:hyp}.
Hence the backward branch of $\alpha'$ does not meet $\Gamma$.

By Lemma~\ref{lem:typeII}, $\theta\to0$ and $\varphi\to\varphi_m\ge0$. If
$\varphi_m=0$, then $\alpha'$ would converge to the corner, so by
Proposition~\ref{prop:corner} it would be $\Delta$, which is impossible since
$\alpha'$ lies strictly above $\Delta$. So $\varphi_m>0$, and
Lemma~\ref{lem:typeII} gives that $s_m$ and the limit of $u$ are finite and that
$\kappa\to+\infty$.

Applying $S$ again, the end of $\alpha$ has
$\varphi_M=\pi-\varphi_m\in(0,\pi)$. The hyperbolic length is finite because
$f(\varphi_M)=\sin\varphi_M>0$, and the expansion follows as in the proof of
Theorem~\ref{th:upper}, which uses only that $f(\varphi_M)>0$ and that $h$ is
finite near $\varphi_M$.
\end{proof}

\begin{corollary}\label{cor:complete}
The horocycles $y=b(x^2+y^2)$, $b>0$, of the hyperbolic class are, among the
solitons of the three classes, the only complete ones and the only ones of
constant curvature.
\end{corollary}

\begin{proof}
In the parabolic and elliptic cases every maximal soliton has a singular
endpoint at finite hyperbolic distance (Theorem~\ref{th:upper}). In the
hyperbolic case the same holds for every soliton other than the horocycles, by
Theorem~\ref{th:hypends}(2) and Lemma~\ref{lem:sym}. A horocycle is complete.
The assertion on constant curvature is Proposition~\ref{prop:cc} and
Theorem~\ref{th:horo}.
\end{proof}

\begin{remark}[Continuation across singular endpoints]\label{rem:cusps}
The regular system \eqref{eq:desing}, together with
$dq/d\tau=\sin\theta\cos\theta$, is defined for all $\theta\in\R$, is
$\pi$-periodic in $\theta$ and satisfies $d\theta/d\tau=1$ on
$\{\theta\in\pi\mathbb Z\}$. So its orbits cross the lines $\theta=k\pi$
transversally. The symmetry $(p(s),\theta(s))\mapsto(p(-s),\theta(-s)-\pi)$ of
\eqref{sys2}, which reverses the orientation of $\gamma$, identifies each strip
$\{(k-1)\pi<\theta<k\pi\}$ with $\Theta$. Hence an orbit of \eqref{eq:desing}
through a point $(p_\ast,k\pi)$ joins two maximal solitons. The first one has a
singular endpoint with $\theta\to\pi$ and $p_M=p_\ast$. The second one has a
singular endpoint with $\theta\to0$ and $p_m=p_\ast$. Both endpoints are the same
point $(q_\ast,p_\ast)$ of $\h^2$. By Theorem~\ref{th:upper} and
Lemma~\ref{lem:typeII}, near it the two arcs are
$p=p_\ast\mp C(q-q_\ast)^{3/2}+o\big((q-q_\ast)^{3/2}\big)$, $q>q_\ast$, so
together they form a semicubical cusp, as two consecutive arches of the cycloid.

In the parabolic and elliptic cases, the soliton that continues a singular
endpoint has $p_m=p_\ast>0$, so it is of Type~II, and its other end is again a
singular endpoint by Theorem~\ref{th:upper}. Therefore every soliton continues
forwards, through infinitely many cusps, along a chain of Type~II arcs. In the hyperbolic
case, the soliton that continues the singular endpoint of a soliton with
$\varphi_0<\pi/2$ has an end with $\theta\to0$ at $\varphi_\ast>0$, so it is not
a horocycle and, by Theorem~\ref{th:hyp}, it has $\varphi_0>\pi/2$. Its other end
lies on the ideal boundary. The resulting curve has exactly one cusp and both
ends on the ideal boundary. In the picture of \cite{nsr}, the cusps correspond
to the zeros of the curvature of the dual curve in the de Sitter plane, since
that curvature is, up to sign, $1/\kappa$.
\end{remark}

\begin{figure}[ht]
\centering
\includegraphics[width=\textwidth]{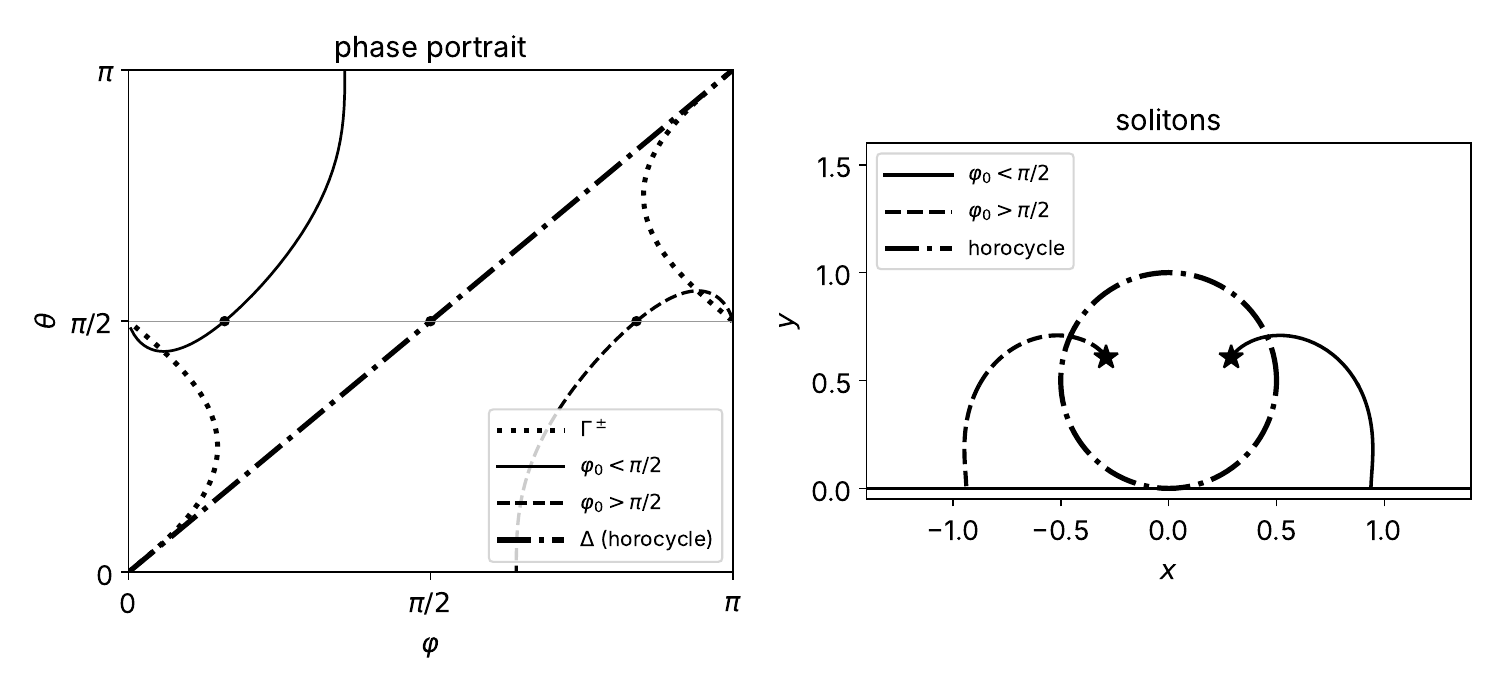}
\caption{The hyperbolic case. Left: phase portrait of \eqref{hypsys}. The dotted
curves are $\Gamma^-$ and $\Gamma^+$, and the dots mark the points with
$\theta=\pi/2$. The solid orbit has $\varphi_0<\pi/2$, the dashed one is its image
under the symmetry $S$ of Lemma~\ref{lem:sym}, and the dash-dotted diagonal
$\Delta$ is the critical orbit. Right: the corresponding solitons in the upper
half-plane. Stars mark the singular endpoints. The dashed soliton is the
reflection of the solid one in the $y$-axis, and the dash-dotted curve is the
horocycle $y=x^2+y^2$. Distinct orbits do not meet in the phase plane, but the
corresponding solitons may intersect in $\h^2$, as the solid soliton and the
horocycle do here.}
\label{fig:hyp}
\end{figure}

\begin{remark}\label{rem:c}
Let $X=cX_{\mathrm{hyp}}$ or $cX_{\mathrm{ell}}$ with $c>0$. The computation of
Section~\ref{sec:unified} gives $\theta'=\frac{1}{c\sin\theta}-h(p)\cos\theta$,
and $c$ cannot be removed, because the translation length and the rotation speed
are invariant under conjugation. We have taken $c=1$. For this value the
family $\gamma_t$ of Proposition~\ref{prop:evol}(2) is the orbit of $\gamma_0$
under the group $z\mapsto e^tz$, so the horocycles with ideal point $0$ are
solitons. For $c\neq1$ the diagonal
$\{\varphi=\theta\}$ is no longer an orbit. In the parabolic class $c$ is
absorbed by a homothety, and the critical value is $H(c)=cH(1)$. In the elliptic
class the results on the two ends and the bounds for $H$ adapt, with
$\Gamma_c=\{c\,h\sin\theta\cos\theta=1\}$, $p^\ast_c=\artanh(c/2)$ and
$L_c(p_1)=\frac{c\pi/4}{1-c\,h(p_1)/2}$, provided $c<2$. For $c\ge2$ the set
$\Gamma_c$ meets every level of $p$, because $\coth p>1$, and the arguments above
do not apply. We do not pursue the case $c\neq1$.
\end{remark}

%%%%%%%%%%%%%%%%%%%%%%%%%%%%%%%%%%%%%%%%%%%%%%%%%%%%%%%%%%%%%%%%%%%%%%%%%%%%
\section{A conformal field}\label{sec:conf}

Proposition~\ref{prop:soliton} uses
that $\Phi_t$ is an isometry. The field $\partial_p$ of $M_f$ is conformal but
not Killing, since $\mathcal L_{\partial_p}\langle,\rangle=-2h\langle,\rangle$. In
the half-plane with $f(p)=p$ it is $\partial_y$. In the elliptic model,
$\rho=2\operatorname{artanh}(e^{-t})$ gives $\partial_t=-\sinh\rho\,\partial_\rho$,
which is, up to sign, the radial conformal field $V=\phi(r)\partial_r$ of
\cite{kw} with $K=-1$. Its flow
$\Phi_t(q,p)=(q,p+t)$ is not an isometry, and a soliton must be defined by
requiring that the whole family $\Phi_t\circ\gamma$ solves \eqref{flow}.

\begin{proposition}\label{prop:conf}
Let $f(p)=p$, $\sin p$ or $\sinh p$. There is no regular curve $\gamma$ in $M_f$
such that $\Phi_t\circ\gamma$ solves \eqref{flow} for $t$ in an interval, where
$\Phi_t$ is the flow of $\partial_p$.
\end{proposition}

\begin{proof}
Parametrise $\gamma$ by Euclidean arc length in the $(q,p)$-plane, with angle
$\theta(s)$. Since $\Phi_t$ is a Euclidean translation, $\Phi_t\circ\gamma$ has
the same $\theta(s)$ and its points are $(q(s),P)$ with $P=p(s)+t$. By
\eqref{eq:Nk}, $\langle\partial_t\gamma_t,N\rangle=\langle\partial_p,N\rangle
=\cos\theta/f(P)$ and $\kappa=f(P)\theta'+f'(P)\cos\theta$. Then \eqref{flow}
reads
\[
f(P)\,\big(1+\theta'\cos\theta\big)+f'(P)\cos^2\theta=0 .
\]
Fix $s$ and let $t$ vary, so that $P$ runs over an interval. The functions $f$ and
$f'$ are linearly independent there, being $p$ and $1$, $\sin p$ and $\cos p$, or
$\sinh p$ and $\cosh p$. Hence $\cos^2\theta=0$ and $1+\theta'\cos\theta=0$, which
is impossible.
\end{proof}

The solutions of the pointwise equation $1/\kappa=-\langle N,\partial_p\rangle$,
analogous to the equation of conformal solitons of the mean curvature flow in
hyperbolic space \cite{moss}, do not move under \eqref{flow} by the flow of
$\partial_p$, by Proposition~\ref{prop:conf}. We do not consider them here.

%%%%%%%%%%%%%%%%%%%%%%%%%%%%%%%%%%%%%%%%%%%%%%%%%%%%%%%%%%%%%%%%%%%%%%%%%%%%
\section*{Ethics declarations}

The authors declare that they have no conflict of interest. The authors have no
relevant financial or non-financial interests to disclose. No datasets were
generated or analysed during the current study.

%%%%%%%%%%%%%%%%%%%%
 \section*{Acknowledgements}
Rafael L\'opez has been partially supported by MINECO/MICINN/FEDER grant no. PID2023-150727NB-I00, and by the ``Mar\'{\i}a de Maeztu'' Excellence Unit IMAG, reference CEX2020-001105-M, funded by MCIN/AEI/10.13039/501100011033.
%%%%%%%%%%%%%%%%%%%%%%%%

%%%%%%%%%%%%%%%%%%%%%%%%%%%%%%%%%%%%%%%%%%%%%%%%%%%%%%%%%%%%%%%%%%%%%%%%%%%%


\begin{thebibliography}{99}

\bibitem{bda} B. D. Allen, Non-compact solutions to inverse mean curvature flow in hyperbolic space. Ph.D. thesis, University of Tennessee, Knoxville, 2016.

\bibitem{bl1} A. Bueno, R. L\'opez, Horo-shrinkers in the hyperbolic space. Taiwanese J. Math. 29 (2025), 1037--1059.

\bibitem{bl2} A. Bueno, R. L\'opez, The class of grim reapers in $H^2\times R$. J. Math. Anal. Appl. 541 (2025), 128730.

\bibitem{ca} J. Carr, Applications of centre manifold theory. Applied Mathematical Sciences, vol. 35, Springer-Verlag, New York-Berlin, 1981.

\bibitem{cl} I. Castro, A. M. Lerma,  Lagrangian homothetic solitons for the inverse mean curvature flow. Results Math. 71 (2017), 1109--1125.

\bibitem{cd} B. Choi, P. Daskalopoulos, Evolution of non-compact hypersurfaces by inverse mean curvature. Duke Math. J. 170 (2021), 2755--2803.

\bibitem{ch} B. Choi, P.K. Hung, Inverse mean curvature flow with singularities. Int. Math. Res. Not. IMRN 10 (2023), 8683--8702.

\bibitem{chg} J. E. Chang, Self-similar solutions to the inverse mean curvature flow in $\mathbb R^2$. Results Math. 75 (2020), 175.

\bibitem{dhw} G. Drugan, H. Lee, G. Wheeler, Solitons for the inverse mean curvature flow. Pacific J. Math.  284 (2016), 309--326.

\bibitem{ga} M. E. Gage, Curve shortening makes convex curves circular. Invent. Math.  76 (1984), 357--364.

\bibitem{ge} C. Gerhardt, Flow of nonconvex hypersurfaces into spheres. J.   Differential Geom. 32 (1990), 299--314.

\bibitem{ge2} C. Gerhardt, Curvature problems. Ser. Geom. Topol. vol. 39. International Press, Boston. 2006.

\bibitem{ge3} C. Gerhardt, Inverse curvature flows in hyperbolic space. J. Differential Geom. 89 (2011), 487--527.

\bibitem{hi} G. Huisken, T. Ilmanen, The inverse mean curvature flow and the Riemannian Penrose inequality. J.    Differential Geom. 59 (2001), 353--437.

\bibitem{hw} P.-K. Hung, M.-T. Wang, Inverse mean curvature flows in the hyperbolic 3-space revisited. Calc. Var. Partial Differential Equations 54 (2015),   119--126.

\bibitem{kp} D. Kim, J. Pyo, Remarks on solitons for inverse mean curvature flow. Math. Nachr. 293 (2020), 2363--2369.

\bibitem{kw} K.-K. Kwong, Y. Wei, G. Wheeler, V.-M. Wheeler, On an inverse curvature flow in two-dimensional space forms. Math. Ann. 384 (2022), 1--24.

\bibitem{lw} H. Li, Y.  Wei, On inverse mean curvature flow in Schwarzschild space and Kottler space. Calc. Var. Partial Differential Equations 56 (2017), no. 3, Paper No. 62.

\bibitem{lu} S. Lu, Inverse curvature flow in anti-de Sitter-Schwarzschild manifold. Comm. Anal. Geom. 27 (2019), 465--489.

\bibitem{moss} L. Mari, J. R. Oliveira, A. Savas-Halilaj, R. Sod\'re de Sena, Conformal solitons for the mean curvature flow in hyperbolic space. Ann. Global Anal. Geom. 65 (2024), 19.

\bibitem{nt} F. Nunes da Silva, K. Tenenblat, Soliton solutions to the curve shortening flow on the 2-dimensional hyperbolic space. Rev. Mat. Iberoam. 38 (2022), 1763--1782.

\bibitem{dst} F. Nunes da Silva, K. Tenenblat, Self-similar solutions to the curvature flow and its inverse on the 2-dimensional light cone. Ann. Mat. Pura Appl. 202 (2023), 253--285.

\bibitem{nsr} F. Nunes da Silva, E. Salinas Reyes, Soliton solutions to the curvature flow on the 2-dimensional de Sitter space and applications. arXiv:2608.27093 (2026).

\bibitem{pi1} G. Pipoli, Inverse mean curvature flow in quaternionic hyperbolic space. Atti Accad. Naz. Lincei Rend. Lincei Mat. Appl. 29 (2018), 153--171.

\bibitem{pi2} G. Pipoli,  Inverse mean curvature flow in complex hyperbolic space. Ann. Scient. \'{E}c. Norm. Sup. 52 (2019),  1107--1135.

\bibitem{sc} J. Scheuer, The inverse mean curvature flow in warped cylinders of non-positive radial curvature. Adv. Math.   306 (2017), 1130--1163.

\bibitem{ns} G. Silva Neto, V. Silva,   Classification of ruled surfaces as homothetic self-similar solutions of the inverse mean curvature flow in the Lorentz-Minkowski 3-space. Bull. Braz. Math. Soc. New Ser. 54, Art. 48 (2023).

\bibitem{ti} X. Tian, Soliton solutions to the inverse curvature flow on the de Sitter surface $dS_2$. AIMS Math. 11 (2026), 18885--18916.

\bibitem{ur} J. I. E. Urbas, On the expansion of starshaped hypersurfaces by symmetric functions of their principal curvatures. Math. Z. 205 (1990), 355--372.

\bibitem{wx} E. Woolgar, R. Xie, Self-similar curve shortening flow in hyperbolic 2-space. Proc. Amer. Math. Soc. 150 (2022), 1301--1319.

\bibitem{zh} H. Zhou,   Inverse mean curvature flows in warped product manifolds. J. Geom. Anal. 28 (2018), 1749--1772.

\end{thebibliography}
\end{document}